\documentclass[letterpaper,12pt]{article}
\pdfoutput=1
\usepackage{graphicx}
\usepackage{latexsym}
\usepackage[left=1.5in,top=1.5in,right=1.5in,bottom=1.5in]{geometry}
\usepackage{amsmath}
\usepackage{bbm}
\usepackage{booktabs}
\usepackage{color}
\usepackage{graphicx} 
\usepackage{lscape} 
\usepackage{tikz}
\usetikzlibrary{shapes}
\usetikzlibrary{arrows}
\usepackage{amsthm}

\newtheorem{theorem}{Theorem}

\newtheorem{lemma}{Lemma}
\usepackage{float}
\restylefloat{table}

\usepackage[ruled,vlined]{algorithm2e}
\usepackage[noend]{algpseudocode}
\newcommand{\openr}{\hbox{${\rm I\kern-.2em R}$}}
\newcommand{\openn}{\hbox{${\rm I\kern-.2em N}$}}

\usepackage{authblk}
\usepackage[sorting=none]{biblatex}

\begin{document}

\title{Estimation of population size based on capture recapture designs and evaluation of the estimation reliability}

\author[1]{Yue You}
\author[1]{Mark van der Laan}
\author[2]{Philip Collender}
\author[2]{Qu Cheng}
\author[1]{Alan Hubbard}
\author[1]{Nicholas P Jewell}
\author[1]{Zhiyue Tom Hu}
\author[3]{Robin Mejia}
\author[2]{Justin Remais}
\affil[1]{Division of Biostatistics and Epidemiology, University of California, Berkeley}
\affil[2]{Division of Environmental Health Sciences, University of California, Berkeley}
\affil[3]{Department of Statistics, Carnegie Mellon University}

\date{\today}

\maketitle

\begin{abstract}
We propose a modern method to estimate population size based on capture-recapture designs of K samples. The observed data is formulated as a sample of n i.i.d. K-dimensional vectors of binary indicators, where the k-th component of each vector indicates the subject being caught by the k-th sample, such that only subjects with nonzero capture vectors are observed. The target quantity is the unconditional probability of the vector being nonzero across both observed and unobserved subjects. We cover models assuming a single general constraint on the K-dimensional distribution such that the target quantity is identified  and the statistical model is unrestricted. We present solutions for general linear constraints, as well as constraints commonly assumed to identify capture-recapture models, including no K-way interaction in linear and log-linear models, independence or conditional independence. We demonstrate that the choice of constraint(identification assumption) has a dramatic impact on the value of the estimand, showing that it is crucial that the constraint is known to hold by design. For the commonly assumed constraint of no K-way interaction in a log-linear model, the statistical target parameter is only defined when each of the  $2^K - 1$ observable capture patterns is present, and therefore suffers from the curse of dimensionality. We propose a targeted MLE based on undersmoothed lasso model to smooth across the cells while targeting the fit towards the single valued target parameter of interest. For each identification assumption, we provide simulated inference and confidence intervals to assess the performance on the estimator under correct and incorrect identifying assumptions. We apply the proposed method, alongside existing estimators, to estimate prevalence of a parasitic infection using multi-source surveillance data from a region in southwestern China, under the four identification assumptions. 
\end{abstract}


\section{Introduction}
Epidemiologists use surveillance networks to monitor trends in disease frequency. When multiple surveillance components or surveys gather data on the same underlying population (such as those diagnosed with a particular disease over a particular time period), a variety of methods (capture-recapture designs \cite{chao2001applications}, distance sampling \cite{buckland2001introduction}, multiple observers \cite{alldredge2006estimating}, etc.) may be used to better estimate the disease occurrence in the population. Capture-recapture models are widely used for estimating the size of partially observed populations, usually assuming that individuals do not enter or leave the population between sample collections  \cite{chao2001overview, kurtz2013local}. These models have been widely applied to epidemiological data \cite{wittes1974applications}.

Due to the unobservability of outcomes for individuals not captured by any survey, additional identifying assumptions have to be made in capture-recapture problems. In two-sample scenarios, a common identification assumption is independence between the two samples (i.e. that capture in one survey does not change the probability of capture in the other survey). The estimator of population size based on this assumption is known as the Lincoln-Petersen estimator \cite{seber1982estimation}. However, the independence assumption is often violated in empirical studies. In problems involving three or more surveys, it is common to assume that the highest order interaction term in log-linear or logistic model equals zero (i.e., that correlations between survey captures can be described by lower-order interactions alone), an assumption that is very difficult to interpret or empirically verify \cite{chao2001overview}. An additional challenge to capture-recapture estimators is the curse of dimensionality in the finite sample case, whereby the absence of one or more capture patterns from the observed sample leads to undefined interaction terms. Traditionally, a common approach to selecting among alternative capture-recapture models is to perform model selection based on Akaike's Information Criterion (AIC) \cite{bozdogan1987model}, Bayesian Information Criterion (BIC) \cite{schwarz1978estimating}, or Draper's version of the Bayesian Information Criterion \cite{draper1995assessment}. However, this approach is known to have limited reliability in the presence of violations of its identifying assumptions \cite{chao2001overview, hook1997validity, braeye2016capture}. Das and Kennedy made contributions on a doubly robust method under a specific identification assumption that two lists are conditionally independent given measured covariate information \cite{das2021doubly}. 

In this paper, we propose a generalizable framework for estimating population size with as few model assumptions as possible. This framework can be adapted to posit various identification assumptions, including independence between any pair of surveys or absence of highest-order interactions, and can be applied to linear and nonlinear constraints. In high dimensional settings with finite samples, we use machine learning methods to smooth over unobserved capture patterns and apply targeted maximum likelihood estimation (TMLE) \cite{van2011targeted} updates to reduce estimation bias. Previous work has shown the vulnerability of the existing estimators with violations of identification assumptions \cite{grimm2014reliability, rees2011testing}, and we further show the significant impact of the misspecified identification assumption on the estimation results for all existing and proposed estimators. 

\subsection{Paper outline}
In this paper, we start with the statistical formulation of the estimation problem. In section \ref{formulation}, we define the framework of our estimators under linear and non-linear constraints. Specifically, we develop the estimators under each of the following identification assumptions: no K-way interaction in a linear model, independence among samples, conditional independence among samples, and no K-way interaction in a log-linear model. In section \ref{efficient}, we derive the efficient estimator of the target parameter, and the statistical inference. In section \ref{TMLE} we provide the targeted learning updates for the smoothed estimators under the no K-way interaction log-linear model. In section \ref{simulation}, we illustrate the performance of existing and proposed estimators in various situations, including high-dimensional finite sample settings. In section \ref{sensitivity}, we show the performance of the estimators under violations of various identification assumptions. In section \ref{data}, we apply the estimators to surveillance data  on a parasitic infectious disease in a region in southwestern China \cite{spear2004factors}. In section \ref{conclusions}, we summarise the characteristics of the proposed estimators and state the main findings of the paper.

\clearpage
\section{Statistical formulation of estimation problem}\label{formulation}
\subsection{Defining the data and its probability distribution}
We define the capture-recapture experiments in the following manner. One takes a first random sample from the population of size $n_1$, records identifying characteristics of the individuals captured and an indicator of their capture, then repeats the process a total of $K$ times, resulting in $K$ samples of size $n_1,\ldots,n_K$. 

Each individual $i$ in the population, whether captured or not, defines a capture history as a vector $B^*_i =(B^*_i(1), \ldots, B^*_i(K))$, where $B^*_i(k)$ denotes the indicator that this individual $i$ is captured by sample $k$. We assume the capture history vectors $B^*_i$ of all $i = 1, \dots, N$ individuals independently and identically follow a common probability distribution, denoted by $P_{B^*}$. Thus $P_{B^*}$ is defined on $2^K$ possible vectors of dimension $K$.

Note that, for the individuals contained in our observed sample of size $n = \sum_{k=1}^K n_k$, we actually observe the realization of $B^*$.
For any individual $i$ that is never captured, we know that $B^*_i=0$, where $0=(0,\ldots,0)$ is the $K$ dimensional vector in which each component equals $0$. However, for these $N-n$ individuals we do not know the identity of these individuals and we also do not know how many there are (i.e., we do not know $N-n$). 

Therefore, we can conclude that our observed data set $B_1,\ldots,B_n$ are $n$ independent and identically distributed draws from the conditional distribution of $B^*$, given $B^*\not =0$. Let's denote this true probability distribution with $P_0$ and its corresponding random variable with $B$.

So we can conclude that under our assumptions we have that $B_1,\ldots,B_n\sim_{iid} P_0$, where $P_0(b)=P_{B^*,0}(b\mid B\not =0)$ for all $b\in \{0,1\}^K$. Since the probability distribution $P$ of $B$ is implied by the distribution $P^*$ of $B^*$ we also use the notation 
$P=P_{P^*}$ to stress this parameterization.

\subsection{Full-data model, target quantity}
Let ${\cal M}^F$ be a model for the underlying distribution $P_{B^*,0}$. The full-data target parameter $\Psi^F:{\cal M}^F\rightarrow\openr$ 
is defined as \[
\Psi^F(P_{B^*})=P_{B^*}(B\not =0).\] In other words, we want to know the proportion of $N$ that on average will not be caught by the combined sample. An estimator $\psi^F_n$ of this $\psi^F_0=\Psi^F(P_{B^*,0})$ immediately translates into an estimator of  the desired size $N$ of the population of interest: $N_n=\frac{n}{\psi^F_n}$.
We will focus on full-data models defined by a single constraint, where we distinguish between linear constraints defined by a function $f$ and a non-linear constraint defined by a function $\Phi$. 
Specifically,  for a given function $f$, we define the full-data model
\[
{\cal M}^F_f=\left\{P_{B^*}: \sum_{b}f(b)P_{B^*}(b)=0, 0<P^*(0)<1\right\}.\]
We will require that $f$ satisfies the following condition:
\begin{equation}\label{keyf}
f(b=0)\not = 0.
\end{equation}

In other words, ${\cal M}^F_f$ contains all probability distributions of $B^*$ for which $E_{P_{B^*}}f(B^*)=0$.
An example of interest is:
\[
f_I(b)=(-1)^{K+\sum_{k=1}^K b_k}.\] In this case, $E_0f_{I}(B^*)=0$ is equivalent with
\begin{eqnarray}
\sum_b (-1)^{K+\sum_{k=1}^K b_k} P_{B^*,0}(b)=0.\label{linear_constraint_equation}
\end{eqnarray}
We note the left-hand side represents a $K$-th way interaction term $\alpha_{1}$ in the saturated model
\[
P_{B^*}(b)=\alpha_0+\sum_{b'\not =0} \alpha_{b'}\prod_{j:b_j'=1}b_j.\]
For example, if $K=2$, then the latter model states:
\[
P_{B^*}(b_1,b_2)=\alpha_{00}+\alpha_{10}b_1+\alpha_{01}b_2+\alpha_{11}b_1b_2,\]
and the constraint $Ef_{I}(B^*)=0$ states that $\alpha_{11}=0$.

One might also use a log-link in this saturated model:
\[
\log P_{B^*}(b)=a_0+\sum_{b'\not =0} a_{b'}\prod_{j:b_j'=1}b_j.\]
In this case, $a_{1}$ is the $K$-way interaction term in this log-linear model, and we now have
\[
a_{1}\equiv \sum_b (-1)^{K+\sum_{k=1}^K b_k} \log P_{B^*,0}(b)=0.\]
So, assuming $a_1=0$ corresponds with assuming
\[
0=\sum_b (-1)^{K+\sum_{k=1}^K b_k} \log P_{B^*,0}(b).\]
This is an example of a  non-linear constraint $\Phi_I(P^*)=0$, where
\begin{eqnarray}
\Phi_I(P^*)\equiv \sum_b (-1)^{1+\sum_{k=1}^K b_k} \log P_{B^*}(b).
\end{eqnarray}\label{kway_constraint}
We will also consider general non-linear constraints defined by such a function $\Phi$, so that for that purpose one can keep this example $\Phi_I$ in mind. As we will see the choice of this constraint has quite dramatic implications on the resulting statistical target parameter/estimand and thereby  on the resulting estimator.
As one can already tell from the definition of $\Phi_I$, $\Phi_I$ is not even defined if there are some $b$ for which $P_{B^*}(b)=0$, so that also an NPMLE of $\Psi^F_{\Phi_I}(P^*_0)$ will be ill defined in the case that the empirical distribution $P_n(B=b)=0$ for some $b\not =0$. 
As we will see the parameter $\Psi^F_{f_I}(P^*_0)$ is very well estimated by the MLE, even when $K$ is very large relative to $n$, but for $\Psi^F_{\Phi_I}(P^*_0)$ we would  need a so called TMLE, incorporating machine learning \cite{van2011targeted}.

Since it is hard to believe that the $\Phi_I$ constraint is more realistic than the $f$-constraint in real applications, it appears that, without a good reason to prefer $\Phi_I$, the $f$-constraint is far superior. However, it also raises alarm bells that the choice of constraint, if wrong, can result in dramatically different statistical output, so that one should really try to make sure that the constraint that is chosen is known to hold by design.

Another example of a non-linear constraint is the independence between two samples, i.e., that \[
P^*(B^*(1:2)=(0,0))=P^*(B^*(1)=0)P^*(B^*(2)=0),\] i.e., that the binary indicators $B^*(1)$ and $B^*(2)$ are independent, but the remaining components can depend on each other and depend on $B^*(1),B^*(2)$. 
This corresponds with assuming $\Phi_{II}(P^*)=0$, where
\begin{equation}
\Phi_{II}(P^*)\equiv \sum_b I(b(1:2)=(0,0))P^*(b)-\sum_{b_1,b_2}I(b_1(1)=b_2(2)=0)P^*(b_1)P^*(b_2).\label{PhiII}
\end{equation}

A third non-linear constraint example is conditional independence between two samples, given the others. Suppose we have $K$ samples in total, and the distribution of the $j^{th}$ sample $B_j$ is independent of the $m^{th}$ sample $B_m$ given all other samples (there's no time ordering of $j, m$). Then we can derive the target parameter and efficient influence curve for this conditional independence constraint. The conditional independence constraint $\Phi_{CI} = 0$ is defined as\\
\begin{eqnarray*}
&\Phi_{CI,(j,m)} =
P^*(B_j =1 \vert B_1 = b_1, \cdots, B_m = 0, \cdots, B_K = b_k) \\
&- P^*(B_j =1 \vert B_1 = b_1, \cdots, B_m = 1, \cdots, B_K = b_k), b_t = 0,1, t = 1,\cdots, K.
\end{eqnarray*}
Because we must have the term $P^*(0,0,\cdots, 0)$ in the constraint to successfully identify the parameter of interest, the constraint $\Phi_{CI} = 0$ is sufficient. Thus the equation above can be presented as \\
\begin{eqnarray}
&\Phi_{CI,(j,m)} =
P^*(B_j =1 \vert B_1 = 0, \cdots, B_m = 0, \cdots, B_K = 0) \nonumber \\
&- P^*(B_j =1 \vert B_1 = 0, \cdots, B_m = 1, \cdots, B_K = 0).\label{cond_indep_constraint}
\end{eqnarray}
This is because for all the combinations of $\{B_1 = b_1, \cdots, B_{m - 1} = b_{m - 1}, B_{m + 1} = b_{m + 1}, \cdots, B_K = b_K\}, \forall b_t \in \{0, 1\}, t = 1,\cdots, K $, only the situation of $\{B_1 = 0, \cdots, B_{m - 1} = 0, B_{m + 1} = 0, \cdots, B_K\ = 0\}$ can generate the required term $P^*(0,0,\cdots, 0)$ in the constraint. \\


\subsection{Identifiability from probability distribution of data}\label{indep_define}\label{linear_estimand}
Given such a full-data model with one constraint defined by $f$ or $\Phi$, one will need to establish that for all $P^*\in {\cal M}^F_f$, we have (say we use $f$)
\[
\Psi^F(P^*)=\Psi_f(P).\]
for a known $\Psi_f:{\cal M}\rightarrow (0,1)$ and $P=P_{P^*}$, where the statistical model for $P_0$ is defined as
\[
{\cal M}_f=\{P_{P_{B^*}}: P_{B^*}\in {\cal M}^F_f\}.\]

Let's first study this identifiability problem in the special case of our full data models defined by the linear constraint $P^*_0f=0$ with $f(0)\not =0$ (equation \ref{linear_constraint_equation}).
It will show that we have identifiability of the whole $P_{B^*}$ from $P=P_{P_{B^*}}$.
Firstly, we note that $P_{P^*}(b)=P^*(b)/\psi^F$ for all $b\not =0$. Thus, $P^*(b)=\psi^F P(b)$ for all $b\not =0$.
We also have $P^*(0)=1-\psi^F$, so that $\sum_b f(b)P^*(b)=0$ yields:
\begin{eqnarray*}
0&=&\sum_b f(b)P^*(b)=\sum_{b\not =0} f(b) \psi^F P(b)+f(0)P^*(0)\\
&=&\psi^F \sum_{b\not =0} f(b)P(b)+f(0)(1-\psi^F).
\end{eqnarray*}
We can now solve for $\psi^F$:
\[
\psi^F=\frac{f(0)}{f(0)-\sum_{b\not =0} f(b)P(b)}.\]
At first sight, one might wonder if this solution is in the range $(0,1)$.
Working backwards from the right-hand side, i.e., using  $f(0)P^*(0)+\sum_{b\not =0}f(b)P^*(b)=0$ and $P^*(b)=P(b)/\psi^F$,  it indeed follows that the denominator equals $f(0)+f(0)(1-\psi^F)/\psi^F$, so that the right-hand side is indeed in $(0,1)$. 
Thus, we can conclude that:
\begin{eqnarray}
\Psi^F(P^*)=\Psi_f(P)\equiv \frac{f(0)}{f(0)-Pf},
\end{eqnarray}
where we use the notation $Pf=\int f(b)dP(b)$ for the expectation operator.

Suppose now that our the one-dimensional constraint  in the full-data model is defined in general by 
$\Phi(P^*)=0$ for some function $\Phi:{\cal M}^F_{\Phi}\rightarrow \openr$. The full data model is now defined by 
${\cal M}^F_{\Phi}=\{P^*:\Phi(P^*)=0, 0<P^*(0)<1\}$, and the corresponding observed data model can be denoted with ${\cal M}_{\Phi}$. To establish the identifiability, we still use $P^*(b)=\psi^F P(b)$ for all $b\not =0$. We also still have
$P^*(0)=1-\psi^F$. Define $P^*_{P,\psi}$ as $P^*_{P,\psi}(0)=1-\psi$ and $P^*_{P,\psi}(b)=\psi P(b)$ for $b\not =0$.
The constraint $\Phi(P^*)=0$ now yields the equation
\[
\Phi(P^*_{P,\psi^F})=0\mbox{ in $\psi^F$}\]
for a given $P=P_{P^*}$.

Consider now our particular example $\Phi_I$. 
Note that
\[
\Phi_I(P^*_{P,\psi})=\sum_{b\not =0} (-1)^{K+\sum_k b_k} \log \{\psi P(b)\}+(-1)^K \log (1-\psi).\]
We need to solve this equation in $\psi$ for the given $P$.
For $K$ is odd, we obtain
\[
\Psi_I(P)=\frac{1}{1+\exp\left(\sum_{b\not =0}f_I(b)\log P(b)\right)},\]
and for $K$ even we obtain $\log (1-\psi)/\psi)=\sum_{b\not =0} \log P(b)$ and thus
\[
\Psi_I(P)=\frac{1}{1+\exp\left(-\sum_{b\not =0}f_I(b)\log P(b)\right)}.\]
So, in general, this solution can be represented as:
\begin{equation}
\Psi_I(P)=\frac{1}{1+\exp\left( (-1)^{K+1}\sum_{b\not =0}f_I(b)\log P(b)\right)}.\label{PsiI}
\end{equation}
This solution $\Psi(P)$ only exists if $P(b)>0$ for all $b\not =0$.
In particular, a plug-in estimator $\Psi(P_n)$ based on the empirical distribution function $P_n$ would not be defined
when $P_n(b)=0$ for some cells $b\in \{0,1\}^K$. 

For general $\Phi$, one  needs to assume that this one-dimensional equation $H(\psi^F,P)\equiv \Phi(P^*_{P,\psi})=0$ in $\psi$, for a given $P\in {\cal M}$, 
always has a unique solution, which then proves the desired identifiability of $\psi^F(P^*)$ from $P_{P^*}$ for any $P^*\in {\cal M}^F_{\Phi}$. This solution is now denoted with $\Psi_{\Phi}(P)$. 
So, in this case, $\Psi_{\Phi}:{\cal M}_{\Phi}\rightarrow\openr$ is defined implicitly by $H(\Psi_{\Phi}(P),P)=0$. 
If $\Phi$ is one-dimensional, one will still have that ${\cal M}_{\Phi}$ is nonparametric.

Let's now consider the $\Phi_{II}$ constraint which assumes that $B^*(1)$ is independent of $B^*(2)$.
Again, using $P^*(b)=P(b)\psi$ for $b\not =0$ and $P^*(0)=(1-\psi)$, the equation $\Phi_{II}(P^*)=0$ yields the following quadratic equation in $\psi$:
\[
a_{II}(P)\psi^2+b_{II}(P)\psi =0,\]
where
\begin{eqnarray*}
a_{II}(P)&=& -\sum_{b_1\not =0,b_2\not =0}I(b_1(1)=b_2(2)=0)P(b_1)P(b_2)\\
&&+
\sum_{b_2\not =0}I(b_2(2)=0)P(b_2)+\sum_{b_1\not =0}I(b_1(1)=0)P(b_1)-1\\
b_{II}(P)&=& \sum_{b\not =0}I(b(1:2)=(0,0))p(b)-\sum_{b_2\not =0}I(b_2(2)=0)P(b_2)\\
&&-\sum_{b_1\not =0} I(b_1(1)=0)P(b_1) +1 .
\end{eqnarray*}
Since $\psi\not =0$, this yields the equation $a_{II}(P)\psi+b_{II}(P)=0$ and thus
\[
\Psi_{II}(P)=\frac{-b_{II}(P)}{a_{II}(P)}.\]
A more helpful way this parameter can be represented is given by:
\begin{equation}\label{PsiII}
\Psi_{II}(P)=\frac{ 1-P(B(1)=0)-P(B(2)=0)+P(B(1:2)=(0,0))}{1-P(B(1)=0)-P(B(2)=0)+P(B(1)=0)P(B(2)=0)}.\end{equation}
This identifiability result relies on $\Psi_{II}(P)\in (0,1)$, i..e, that  $P(B(1)=B(2)=0)<P(B(1)=0)P(B(2)=0)$.
In particular, we need $0<P(B(1)=0)<1$ and $0<P(B(2)=0)<1$  and thereby that each of the three cells $(1,0),(0,1),(1,1)$  has positive probability under  the bivariate marginal distribution of $B(1),B(2)$ under $P$.
It follows trivially that the inequality holds for $K=2$ since in that case $P(B(1)=B(2)=0)=0$. It will have to be verified if this inequality constraint always holds  for  $K>2$ as well, or that this is an actual 
assumption in the statistical model. Since it would be an inequality constraint in the statistical model, it would not affect the tangent space and thereby the efficient influence function for $\Psi_{II}:{\cal M}\rightarrow\openr$ presented below, but the MLE would now involve maximizing the likelihood under this constraint so that the resulting MLE of $P_0$ satisfies this constraint.

Similarly, for the conditional independence assumption, the parameter $\Psi_{CI}$ can be derived as
\begin{equation}
\Psi_{CI} = \frac{P(B_m = 1, B_j = 1, 0,...,0)}{C_0(P)},\label{Psi_CI_Eq}
\end{equation}
where 
\begin{eqnarray}
C_0(P)&=& P(B_m = 1, B_j = 1,  0,...,0)\nonumber \\
&& + P(B_m = 0, B_j = 1,  0,...,0)P(B_m = 1, B_j = 0,  0,...,0).\label{C1_P}
\end{eqnarray}

Details on this derivation are presented in the appendix \ref{apd_condInf_psi}.

One could also define a model by a multivariate $\Phi:{\cal M}^F\rightarrow\openr^d$. In this case, the full data model is restricted by $d$ constraints $\Phi(P^*)=0$ and the observed data model will not be saturated anymore.
Let's consider the example in which we assume that $(B^*(1),\ldots,B^*(K))$ are $K$ independent Bernoulli random variables. 
In this case, we assume that $P^*(B^*=b)=\prod_{k=1}^K P^*(B^*(k)=b(k))$ for all possible $b\in \{0,1\}^K$.
This can also be defined as stating that for each 2 components $B^*(j_1),B^*(j_2)$, we have that these two Bernoulli's are independent. 
Let $\Phi_{,II,j_1,j_2}(P)$ be the constraint defined as in (\ref{PhiII}) but with $B(1)$ and $B(2)$ replaced by $B(j_1)$ and $B(j_2)$, respectively.
Then, we can define $\Phi_{III}(P)=(\Phi_{II,j_1,j_2}(P): (j_1,j_2)\in \{1,\ldots,K\}^2,j_1\not =j_2)$,  a vector of dimension $K(K-1)/2$.
This defines now the model ${\cal M}^F_{III}=\{P^*: \Phi_{III}(P^*)=0\}$, which assumes that all components of $B^*$ are independent.
We will also work out the MLE and efficient  influence curve of $\psi_0^F$ for this restricted statistical model.

\subsection{Statistical Model and target parameter}
We have now defined the statistical estimation problem for linear and non-linear constraints. For linear constraints defined by a function $f$, we observe $B_1,\ldots,B_n\sim P_0\in {\cal M}_f=\{P_{P^*}:P^*\in {\cal M}^F\}$, ${\cal M}^F=\{P^*: P^*f=0,0<P^*(0)<1\}$, and our statistical target parameter is given by $\Psi_f:{\cal M}\rightarrow \openr$, where
\begin{eqnarray}
\Psi_f(P)=\frac{f(0)}{f(0)-Pf}.\label{linear_estimand_equation}
\end{eqnarray}

The statistical target parameter satisfies that $\Psi^F(P^*)=P^*(B^*\not =0)=\Psi_f(P_{P^*})$ for all $P^*\in {\cal M}^F$.

Since the full-data model only includes a single constraint, it follows that ${\cal M}_f$ consists of all possible probability distributions of $B$ on $\{b:b\not =0\}$, so that it is a nonparametric/saturated model.

Similarly, we can state the statistical model and target parameter for our examples using non-linear one-dimensional constraints $\Phi$.

In general, we have a statistical model ${\cal M}=\{P_{P^*}:P^*\in {\cal M}^F\}$ for some full data model ${\cal M}^F$ for the distribution of $B^*$, and, we would be given a particular mapping $\Psi:{\cal M}\rightarrow \openr$, satisfying $\Psi^F(P^*)=\Psi(P_{P^*})$ for all $P^*\in {\cal M}^F$. 
In this case, our goal is to estimate $\Psi(P_0)$ based on knowing that $P_0\in {\cal M}$.

 \subsection{Efficient influence curve of target parameter}

An estimator of $\Psi$ is efficient if and only if it is asymptotically linear with influence curve equal to canonical gradient of pathwise derivative of $\Psi$. Therefore it is important to determine this canonical gradient. It teaches us how to construct an efficient estimator of $\Psi(P)$ in model ${\cal M}^F$. In addition, it provides us with Wald type confidence intervals based on an efficient estimator. As we will see for most of our estimation problems $\Psi$ with a nonparametric model, $P$ and $\Psi$ can be estimated with the empirical measure $P_n$ and $\Psi(P_n)$. However for constraint $\Phi_I$, an NPMLE is typically not defined due to empty cells, so that smoothing and bias reduction is needed. 
 
 Let $\Psi_{1f}(P)=\sum_{b\not =0}f(b)P(b)$ so that $\Psi_f(P)=\frac{f(0)}{f(0)-\Psi_{1f}(P)}$.
 Note that $\Psi_{1f}(P)=Pf$ is simply the expectation of $f(B)$ w.r.t  its distribution.
 $\Psi_{1f}:{\cal M}_f\rightarrow\openr$ is pathwise differentiable parameter at any $P\in {\cal M}$ with canonical gradient/efficient influence curve given by:
 \[
 D^*_{1f}(P)(B)=f(B)-\Psi_{1f}(P).\]
 By the delta-method the efficient influence curve of $\Psi_f$ at $P$ is given by:
\begin{eqnarray}
D^*_f(P)(B)=\frac{f(0)}{\{f(0)-\Psi_{1f}\}^2} D^*_{1f}(P)(B).\label{efficient_linear}
\end{eqnarray}

 In general, if $\Psi:{\cal M}\rightarrow\openr$ is the target parameter and the model ${\cal M}$ is nonparametric, then the efficient influence curve 
 $D^*(P)$ is given by $D^*(P)=d\Psi(P)(P_{n=1}-P)$, where 
 \[
d\Psi(P)(h)=\left . \frac{d}{d\epsilon}\Psi(P+\epsilon h) \right |_{\epsilon =0}\] 
is the Gateaux derivative in the direction $h$, and $P_{n=1}$ is the empirical distribution for a sample of size one $\{B\}$, putting all its mass on $B$ \cite{gateaux1919fonctions}.

 Consider now the model ${\cal M}_{\Phi}$ for a general univariate constraint function $\Phi:{\cal M}^F_{NP}\rightarrow\openr$ that maps any possible $P^*$ into a real number. 
 Recall that $\Psi_{\Phi}:{\cal M}_{\Phi}\rightarrow\openr$ is now defined implicitly by $\Phi(P^*_{P,\psi})=0$.
 For notational convenience, in this particular paragraph we suppress the dependence of $\Psi$ on $\Phi$.
  The implicit function theorem implies that the  general form of efficient influence curve is given by:
 \[
 D^*_{\Phi}(P)(B)=-\left \{ \frac{d}{d\psi}\Phi(P^*_{P,\psi})\right\}^{-1}\frac{d}{dP}\Phi(P^*_{P,\psi})(P_{n=1}-P).\]
 where the latter derivative is the directional derivative of $P\rightarrow \Phi(P^*_{P,\psi})$ in the direction $P_{n=1}-P$.
 
   Note that 
 \[
 \frac{d}{d\psi}\Phi(P^*_{P,\psi})=d\Phi(P^*_{P,\psi}) \frac{d}{d\psi}P^*_{P,\psi}.\]
 We now use that $P^*_{P,\psi}(b)=P(b)I(b\not =0)+(1-\psi)I(b=0)$.
 So we obtain:
 \[
  \frac{d}{d\psi}\Phi(P^*_{P,\psi})=d\Phi(P^*_{P,\psi}) (-1I_{0}),\]
  where $I_0(b)$ is the function in $b$ that equals 1 if $b=0$ and zero otherwise, and 
  \[
  d\Phi(P^{*0})(h)\equiv \left . \frac{d}{d\epsilon}\Phi(P^{*0}+\epsilon h)\right |_{\epsilon =0}\]
  is the directional/Gateaux  derivative of 
 $\Phi$ in the direction $h$.  We also have:
  \[ \frac{d}{dP}\Phi(P^*_{P,\psi})(P_{n=1}-P)  =d\Phi(P^*_{P,\psi})\frac{d}{dP}P^*_{P,\psi}(P_{n=1}-P).\]
  We have $\frac{d}{dP}P^*_{P,\psi}(h)=I_{0^c}h$, where $I_{0^c}(b)$ is the function in $b$ that equals zero if $b=0$ and equals 
  $1$ otherwise.
  So we obtain:
  \[
  \frac{d}{dP}\Phi(P^*_{P,\psi})(P_{n=1}-P)  =d\Phi(P^*_{P,\psi})(I_{0^c}(P_{n=1}-P)).\]
  We conclude that
  \[
 D^*_{\Phi}(P)=- \{d\Phi(P^*_{P,\psi}) (-1I_{0})  \}^{-1}d\Phi(P^*_{P,\psi})(I_{0^c}(P_{n=1}-P)).\]
      
 Let's now consider our special example $\Phi_I$. In this case, the statistical target parameter $\Psi_I:{\cal M}\rightarrow\openr$ is given by (\ref{PsiI}). It is straightforward to show that the directional derivative of $\Psi$ at $P=(P(b):b)$ in direction $h=(h(b):b)$ is given by:
 \[
 d\Psi_I(P)(h)=(-1)^K \Psi_I(P)(1-\Psi_I(P))\sum_{b\not =0} \frac{f_I(b)}{P(b)}h(b).\]
 As one would have predicted from the definition of $\Psi_I$, this directional derivative is only bounded if $P(b)>0$ for all $b\not =0$.
 The efficient influence curve is thus given by $d\Psi_I(P)(P_{n=1}-P)$:
 \begin{eqnarray}
D^*_{\Phi_I}(P)&=& (-1)^K \Psi_I(P)(1-\Psi_I(P)) \sum_{b\not =0}\frac{f_I(b)}{P(b)}\{I_{B}(b)-P(b)\} \nonumber \\
 &=& (-1)^K \Psi_I(P)(1-\Psi_I(P)) \left\{ \frac{f_I(B)}{P(B)}+f_I(0) \right\} ,\label{efficphiI}
 \end{eqnarray}

 where we use that $\sum_{b\not =0}f_I(b)=-f_I(0)$ so that indeed the expectation of $D^*_{\Phi_I}(P)$ equals zero (under $P$).
 
 The efficient influence function for $\Psi_{II}:{\cal M}\rightarrow\openr$ (\ref{PsiII}), corresponding with the constraint $\Phi_{II}(P^*)=0$,  is the influence curve of the empirical plug-in estimator $\Psi_{II}(P_n)$, and can thus be derived  from the delta-method:
\begin{eqnarray}
 D^*_{\Phi_{II}}(P)&=&C_2(P)\{I(B(1)=0)-P(B(1)=0)\} \nonumber\\
  &&+ C_3(P)\{I(B(2)=0)-P(B(2)=0)\} \nonumber\\
 && +C_4(P)\{I(B(1)=B(2)=0)-P(B(1:2)=0)\} ,\label{efficphiII}
 \end{eqnarray}
 where
 \[
 \begin{array}{l}
C_2(P)=  \frac{1-P(B(1)=0)-P(B(2)=0)-P(B(1)=B(2)=0)}{(1-P(B(1)=0)-P(B(2)=0)-P(B(1)=0)P(B(2)=0) )^2} P(B(2)=0) \\
C_3(P)=  \frac{1-P(B(1)=0)-P(B(2)=0)-P(B(1)=B(2)=0)}{(1-P(B(1)=0)-P(B(2)=0)-P(B(1)=0)P(B(2)=0) )^2} P(B(1)=0)  \\
C_4(P)=  -\frac{1}{1-P(B(1)=0)-P(B(2)=0)-P(B(1)=0)P(B(2)=0)}   .
\end{array}
\]

Similarly, we can derive the efficient influence curve for conditional independence constraint $\Phi_{CI}$ and parameter $\Psi_{CI}$ under this constraint as 
\begin{eqnarray}\label{efficphiCI}
D^*_{\Phi_{CI}}(P) &=& \frac{1}{C_5(P)}(C_6(P) - C_7(P) - C_8(P)).
\end{eqnarray}

where
\begin{eqnarray*}
C_5(P)&=& -\sum_{b_1\not =0,b_2\not =0}I(b_1(1)=b_2(2)=0)P(b_1)P(b_2).\\
C_6(P)&=& P(B_m = 0, B_j = 1, 0,...,0)P(B_m = 0, B_j = 0, 0,...,0)\\
&& [\mathbbm{I}(B_m = 1, B_j = 1, 0,...,0) - P(B_m = 1, B_j = 1, 0,...,0)].\\
C_7(P) &=& P(B_m = 1, B_j = 0, 0,...,0)P(B_m = 1, B_j = 1, 0,...,0)\\
&& [\mathbbm{I}(B_m = 0, B_j = 1, 0,...,0) - P(B_m = 0, B_j = 1, 0,...,0)].\\
C_8(P) &=& P(B_m = 0, B_j = 1, 0,...,0)P(B_m = 1, B_j = 1, 0,...,0)\\
&& [\mathbbm{I}(B_m = 1, B_j = 0, 0,...,0) - P(B_m = 1, B_j = 0, 0,...,0)].
\end{eqnarray*}
The details for deriving the influence curve $D^*_{\Phi_{CI}}(P)$ can be found in appendix, section \ref{apd_condInf_psi}.

Finally, consider a non-saturated model ${\cal M}$ implied by a multidimensional constraint function $\Phi$. 
In this case, the above $D^*_{\Phi}(P)$ is still a gradient of the pathwise derivative, where division by a vector $x$ is now defined
as $1/x=(1/x_j:j)$ component wise. However, this is now not equal to the canonical gradient.
We can now determine the tangent space of the model ${\cal M}$ and project $D^*_{\Phi}(P)$ onto the tangent space at $P$, which then yields the actual efficient influence curve. We can demonstrate this for the example defined by the multidimensional constraint $\Phi_{III}$ using the general result in appendix \ref{apd_lemma}. 
\clearpage
\section{Efficient estimator of target parameter, and statistical inference}\label{efficient}
Estimation of $\Psi_f(P_0)$ based on statistical model ${\cal M}_f$ and data $B_1,\ldots,B_n\sim_{iid}P_0$ is trivial since
$\Psi_{1f}(P_0)=P_0f$ is just a mean of $f(B)$. In other words, we estimate $\Psi_{1f}(P_0)$ with the NPMLE
\[
\Psi_{1f}(P_n)=P_nf=\frac{1}{n}\sum_{i=1}^n f(B_i),\]
where $P_n$ is the empirical distribution of $B_1,\ldots,B_n$,
and, similarly, we estimate $\Psi_f(P_0)$ with its NPMLE
\[
\Psi_f(P_n)=\frac{f(0)}{f(0)-\Psi_{1f}(P_n)}=\frac{f(0)}{f(0)-P_nf}.\]
This estimator is asymptotically linear at $P_0$ with influence curve $D^*_{\Phi}(P_0)$ under no further conditions.
As a consequence, a valid asymptotic 95\% confidence interval is given by:
\[
\Psi_f(P_n)\pm q(0.975) \sigma_n/\sqrt{n},\]
where
\[
\sigma^2_n\equiv \frac{1}{n}\sum_{i=1}^n \{D^*_f(P_n)(B_i)\}^2,\]
and $q(0.975)$ is the 0.975 quantile value of standard normal distribution. If the general constraint $\Phi:{\cal M}\rightarrow\openr$ is differentiable so that $\Phi(P_n)$ is an asymptotically linear estimator of $\Phi(P_0)$ \cite{gill1989non}, then, we can estimate $\Psi_{\Phi}(P_0)$ with the NPMLE $\Psi_{\Phi}(P_n)$. Again, under no further conditions, $\Psi_{\Phi}(P_n)$ is asymptotically linear with influence curve $D^*_{\Phi}(P_0)$, and an asymptotically valid  confidence interval is obtained as above. 
The estimator $\Psi_f(P_n)$ is always well behaved, even when $n$ is relatively large relative to $K$. For a general $\Phi$, this will very much depend on the the precise dependence on $P$ of $\Phi(P)$. In general, if $\Phi(P_n)$ starts suffering from the dimensionality of the model (i.e., empty cells make the estimator erratic), then one should expect that $\Psi_{\Phi}(P_n)$ will suffer accordingly, even though it will still be asymptotically efficient. In these latter cases, we propose to use targeted maximum likelihood estimation, which targets data-adaptive machine learning fits towards optimal variance-bias trade-off for the target parameter. 
 
Let's consider the $\Phi_I$-example which assumes that the $K$-way interaction on the log-scale equals zero. In this case the target parameter $\Psi_I(P)$ is defined by (\ref{PsiI}), which shows that the NPMLE $\Psi_I(P_n)$ is not defined when $P_n(b)=0$ for some $b\in \{0,1\}^K$. So in this case, the MLE suffers immensely from the curse of dimensionality and can thus not be used when the number $K$ of samples  is such that the sample size $n$ is of the same order as $2^K$. In this context, we discuss estimation based on TMLE in the next section. 

$\Psi_{II}(P_0)$ can be estimated with the plug-in empirical estimator which is the NPMLE. So, this estimator only relies on positive probability on $(0,1),(1,0)$ and $(1,1)$
under the bivariate distribution of $B(1),B(2)$ under $P$. We also note that this efficient estimator of $\psi_{II,0}$ does only use the data on the first two samples. Thus, the best estimator based upon this particular constraint $\Phi_{II}$ ignores the data on all the other samples.
More assumptions will be needed, such as the model that assumes that all components of $B^*$ are independent, in order to create a statistical model that is able to also incorporate the data with the other patterns. 
 
{\bf Constrained models based on multivariate $\Phi$:}
For these models, if the NPMLE would behave well  for the larger model in which just one of the $\Phi$ constraints is used, then it will behave well under more constraints. If on the other hand, this NPMLE suffers from curse of dimensionality, it might be the case that the MLE behaves better due to the additional constraints, but generally speaking that is a lot to hope for (except if $\Phi$ is really high dimensional relative to $2^K$).
To construct an asymptotically efficient estimator one can thus use the MLE $\Psi_{\Phi}(\tilde{P}_n)$ where now $\tilde{P}_n$ is the MLE over the actual model ${\cal M}_{\Phi}$.
In the special case that the behavior of this MLE $\Psi_{\Phi}(\tilde{P}_n)$ suffers from the curse of dimensionality, we recommend the TMLE below instead, which will be worked out for the example $\Phi_I$. 

\section{Targeted maximum likelihood estimation when the (NP)MLE of target parameter suffers from curse of dimensionality}\label{TMLE}
Consider the statistical target parameter $\Psi_{I}:{\cal M}_{\Phi_I}\rightarrow\openr$ defined by (\ref{PsiI}), whose efficient influence curve is given by (\ref{efficphiI}). Let $P_n^0$ be an initial estimator of the true distribution $P_0$ of $B$ in the nonparametric statistical model ${\cal M}_{\Phi_I}$ that consists of all possible distributions of $B=(B(1),\ldots,B(K))$. An ideal initial estimator for $P$ would be consistent and identifiable when empty cells exist. One could use, for example, an undersmoothed lasso estimator constructed in algorithm \ref{algo1} as the initial estimator \cite{van2019efficient}. In algorithm \ref{algo1}, we establish the empirical criterion by which the level of undersmoothing may be chosen to appropriately satisfy the conditions required of an efficient plug-in estimator. In particular, we require that the minimum of the empirical mean of the selected basis functions is
smaller than a constant times $n^{-\frac{1}{2}}$, which is not parameter specific. This condition essentially enforces the selection of the $L_1$-norm in the Lasso to be large enough so that the fit includes sparsely supported basis functions \cite{van2019efficient}. When the hyperparameter $\lambda$ equals zero, the undersmoothed lasso estimator is the same as the NPMLE plug-in estimator, and when $\lambda$ is not zero, the undersmoothed lasso estimator will smooth over the predicted probabilities and avoid predicting probabilities of exact zero when empty cells exist. 

\begin{algorithm}[H]
\SetAlgoLined
The log-linear model can be expressed as:
\[
\log E_{B^*}(b)=a_0+\sum_{b\not =0} a_{b}\prod_{j:b_j=1}b_j.
\]
In this case, $b = (b_1, b_2, ..., b_K)$, and $b_j = 1$ if the subject is captured by sample $j, j = 1,2,..., K$. $E_{B^*}(b)$ is the count of observations in cell $b$. The $K$-way interaction term is $a_{1,1,...,1}$(K terms of 1), and the identification assumption is that $a_{1,1,...,1}=0$.

Fit a lasso regression $M_{lasso}$ with all but the highest way interaction term as the independent variable, $E_{B^*}(b)$ (the count) as the dependent variable, and specify model family as Poisson. The regularization term $\lambda$ is chosen such that the absolute value of the empirical mean of the efficient influence function $P_n D^*_{\Phi_I}(P_n)(B_i) \le T^* = \frac{\sigma_n}{\sqrt n}$, where $\sigma_n = \sqrt{\frac{1}{n}\sum_{i=1}^n \{D^*_{I,cv}(P_n)(B_i)\}^2}$. The probabilities $P_n(B_i), i = 1,...,n$ used in computing $\sigma_n$ are estimated by the lasso regression with regularization term chosen by cross-validation. This algorithms will naturally avoid fits with $P_n(B_i)=0$ for any $B_i$ since that would result in a log-likelihood equal to minus infinity. 

\While{$P_n D^*_{\Phi_I}(P_n)(B_i) \ge T^*$}
{
  1. decrease $\lambda$;\\
  2. predict count $\hat E_{B^*}(b)$ using $M_{lasso}$ and the new $\lambda$;\\
  3. calculate predicted probability $\hat P_B(b)$ by dividing the count $\hat E_{B^*}(b)$ by the sum of all counts;\\
  4. calculate $\hat \Psi_I(P) = \frac{1}{1 + exp((-1)^{K} \sum_{b \ne 0} f_I(b) log \hat P_B(b))}$;\\
  5. calculate $D^*_{\Phi_I}(P_n)(B_i) = \hat \Psi_I(P)(1 - \hat \Psi_I(P))[\frac{f_I(b)}{\hat P_B(b)} + f_I(0)]$;\\
  6. update $P_n D^*_{\Phi_I}(P_n)(B_i) = \frac{1}{n} \sum_{i=1}^n D^*_{\Phi_I}(P_n)(B_i)$.
}
\KwResult{Predicted probability $\hat P_B(b)$ for each cell}
\caption{Undersmoothed Lasso}
\label{algo1}
\end{algorithm}

We denote $P_n^0 \equiv \hat P_B(b)$ as our initial estimator, the TMLE $P_n^*$ will update this initial estimator $P_n^0$ in such a way that $P_n D^*_{\Phi_I}(P_n^*)=0$, allowing a rigorous analysis of the TMLE $\Psi_I(P_n^*)$ of $\Psi_I(P_0)$ as presented in algorithm \ref{algo2}.

Note that indeed this submodel \ref{submodel} $\{P_{\epsilon}:\epsilon\}\subset {\cal M}_{\Phi_I}$ through $P$ is a density for all $\epsilon$ and that its score is given by: 
\[
\left . \frac{d}{d\epsilon}\log P_{\epsilon} \right |_{\epsilon =0}=D^*_{\Phi_I}(P).\]
Recall that $\sum_{b\not = 0} f_I(b)=-f_I(0)$, so that indeed the expectation of its score equals zero:  $E_P D^*_{\Phi_I}(P)(B)=0$. This proves that if we select as loss function of $P$ the log-likelihood loss $L(P)=-\log P$, then the score $\left . \frac{d}{d\epsilon} L(P_{\epsilon}) \right |_{\epsilon =0}$ spans the efficient influence function $D^*_{\Phi_I}(P)$. This property is needed to establish the asymptotic efficiency of the TMLE, with details in appendix section \ref{apd_efficincy_TMLE}.

In accordance with general TMLE procedures \cite{van2011targeted}, the TMLE updating process in algorithm \ref{algo2} will be iterated till convergence at which point $\epsilon_n^m \approx 0$, or $P_n D^*_{\Phi_I}(P_n^m)=o_P(1/\sqrt{n})$. Let $P_n^*=\lim_m P_n^m$ be the probability distribution in the limit. The TMLE of $\Psi_I(P_0)$ is now defined as $\Psi_I(P_n^*)$.
Due to the fact that the MLE $\epsilon_n^m$ solves its score equation and $\epsilon_n^m\approx 0$, it follows that this TMLE satisfies $P_n D^*_{\Phi_I}(P_n^*)=0$ (or $o_P(1/\sqrt{n})$).

\begin{algorithm}[H]
\SetAlgoLined
\textbf{Input}: Vector of predicted probability $\hat P_B(b)$; observed population size $n$;\\
\textbf{Procedure}:\\
1. Denote $P_n^0 = \hat P_B(b)$, calculate $\Psi_n(P_n^0)$ using equation \ref{PsiI}, $D_{\Phi_I}^*(P_n^0)$ using equation \ref{efficphiI}, the empirical mean of $D_{\Phi_I}^*(P_n^0)$ as $P_nD^*_{\Phi_I}(P_n^0) = \frac{1}{n}\sum_{i=1}^n D_{\Phi_I}^*(P_n^0)(B_i)$, and the stopping point $s = \frac{\sqrt{\frac{1}{n}\sum_{i=1}^n D_{\Phi_I}^*(P_n^0)(B_i)^2}}{max(log(n), C)\sqrt n}$, where $C$ is a positive constant.

2. Let $m \in \mathbbm{Z}$ be the number of iterations, and $P_n^m$ is the updated probability in iteration $m$. Initial $m=0$. $\delta$ is a positive constant close to zero.\\

\While{$\lvert P_n D^*_{\Phi_I}(P_n^m)(B_i)\rvert > s$ and $\lvert \epsilon \rvert > \delta$}
{
2.1. calculate the bounds for $\epsilon$ such that the updated probability $\in [0,1]$.
\[ l_\epsilon = max_i[min(-\frac{1}{D^*_{\Phi_I}(P_n^m)(B_i)}, \frac{1-P_n^m(B_i)}{P_n^m(B_i)D^*_{\Phi_I}(P_n^m)(B_i)}))] \]
\[ u_\epsilon = min_i[max(-\frac{1}{D^*_{\Phi_I}(P_n^m)(B_i)}, \frac{1-P_n^m(B_i)}{P_n^m(B_i)D^*_{\Phi_I}(P_n^m)(B_i)}))] \]

2.2. Construct a least favorable parametric model  $\{P_{n,\epsilon}^m:\epsilon\}$ through $P_n^m$ defined as follows:
\begin{equation}\label{submodel}
P_{n,\epsilon}^{m} = C(P_n^m, \epsilon)(1+ \epsilon D^*_{\Phi_I}(P_n^m))P_n^m,
\end{equation}
where
\[
C(P, \epsilon) = \frac{1}{\sum_{b \ne 0} (1+ \epsilon D^*_{\Phi_I}(P(b)))P(b)}, 
\]

2.3. calculate \[ \epsilon^m_n = argmax_\epsilon \frac{1}{n} \sum_{b \ne 0} log(P_{n,\epsilon}^{m}),
\]
where $\epsilon \in [l_\epsilon, u_\epsilon]$.

2.4. $m \gets m + 1$, and update $P_n^m \gets P_{n,\epsilon^m_n}^{m}$.
}

3. Denote $P_n^* = P_n^m$ in the final iteration $m$. Calculate TMLE $\Psi_n(P_n^*)$ using equation \ref{PsiI}, and its efficient influence function $D_{\Phi_I}^*(P_n^*)$ using equation \ref{efficphiI}.

\KwResult{TMLE $\Psi_n(P_n^*)$ and efficient influence function $D_{\Phi_I}^*(P_n^*)$.}
\caption{Targeted maximum likelihood estimation(TMLE) update procedure}
\label{algo2}
\end{algorithm}

The proof of the asymptotic efficiency of TMLE is provided in appendix section \ref{apd_efficincy_TMLE}.


\section{Simulations}\label{simulation}

In this section, we show the performance of all the estimators given that their identification assumptions (linear and non-linear constraints) hold true. The estimand $\Psi_f(P)$ and $\Psi_\Phi(P)$ refer to the probability of an individual being captured at least once under the linear constraint $f$ or non-linear constraint $\Phi$. In subsection \ref{linear_section}, the linear constraint $\Phi_f$ is defined in equation \ref{linear_constraint_equation}, the corresponding estimand (equation \ref{linear_estimand_equation}) \[\Psi_f(P) = \frac{f(0)}{f(0)-Pf},\] and its plug-in estimator \[\Psi_f(P_n) = \frac{f(0)}{f(0)-\sum_{b\not =0} f(b)P_n(b)},\] where $P_n(b)$ is the empirical probability of cell $b$.
In subsection \ref{nonlinear_section}, we provide a summary of all the non-linear constraints defined by equation $\Phi$ and their corresponding estimators. 
In subsection \ref{kwayHolds},  the non-linear constraint $\Phi_I$ is defined by the assumption in equation \ref{kway_constraint}. The estimand (equation \ref{PsiI}) \[\Psi_I(P) = \frac{1}{1+\exp\left( (-1)^{K+1}\sum_{b\not =0}f_I(b)\log P(b)\right)}.\] The plug-in estimator \[\Psi_{I}(P_{NP}) = \frac{1}{1+\exp\left( (-1)^{K+1}\sum_{b\not =0}f_I(b)\log P_n(b)\right)}.\]In addition, the undersmoothed lasso estimator $\Psi_{I}(P_{lasso})$ replaces the plugged-in $P_n(b)$ in $\Psi_{I}(P_{NP})$ with estimated probability in a lasso regression with regularization parameters chosen by the undersmoothing algorithm, the estimator $\Psi_{I}(P_{lasso\_cv})$ replaces the plugged-in $P_n(b)$ with estimated probability in a lasso regression with regularization parameters optimized by cross-validation. The estimator $\Psi_{I}(P_{tmle})$ updates the estimated probabilities of $\Psi_{I}(P_{lasso})$ using targeted learning techniques, and the estimator $\Psi_{I}(P_{tmle\_cv})$ updates the estimated probabilities of $\Psi_{I}(P_{lasso\_cv})$ using targeted learning techniques. In addition, we compare the performance of our proposed estimators to existing estimators $\Psi_{I}(P_{M_0})$ and $\Psi_{I}(P_{M_t})$. 

In subsection \ref{indep_section}, the non-linear constraint $\Phi_{II}$ is defined by the independence assumption in equation \ref{PhiII}. The estimand (equation \ref{PsiII}) \[\Psi_{II}(P) \equiv \frac{1-P(B(1)=0)-P(B(2)=0)+P(B(1:2)=(0,0))}{1-P(B(1)=0)-P(B(2)=0)+P(B(1)=0)P(B(2)=0)},\] and its plug-in estimator \[\Psi_{II}(P_n) \equiv \frac{1-P_n(B(1)=0)-P_n(B(2)=0)+P_n(B(1:2)=(0,0))}{1-P_n(B(1)=0)-P_n(B(2)=0)+P_n(B(1)=0)P_n(B(2)=0)}.\]

In subsection \ref{cond_section}, the non-linear constraint $\Phi_{CI}$ is defined by the conditional independence assumption in equation \ref{cond_indep_constraint}. The estimand (equation \ref{Psi_CI_Eq}, $C_0(P)$ is defined in equation \ref{C1_P}) \[\Psi_{CI}(P) \equiv \frac{P(B_m = 1, B_j = 1, 0,...,0)}{C_0(P)},\] and its plug-in estimator \[\Psi_{CI}(P_n)\frac{P_n(B_m = 1, B_j = 1, 0,...,0)}{C_0(P_n)}.\]

\subsection{Linear identification assumption}\label{linear_section}
The linear identification constraint is $E_{P_{B^*}} f(B^*) = 0$, for any function $f$ such that $f(b=0) \ne 0$. An example of interest is $f(b) = (-1)^{K + \sum_{k=1}^K b_k}$, where K is the total number of samples. In this case, the linear identification assumption is equation \ref{linear_constraint_equation}.
For example, if $K=3$, then the constraint (equation \ref{linear_constraint_equation}) states:
\[
P_{B^*}(b_1,b_2, b_3)=\alpha_{0}+\alpha_{1}b_1+\alpha_{3}b_2+\alpha_{4}b_1b_2 + \alpha_{5}b_1b_3 + \alpha_{6}b_2b_3 + \alpha_{7}b_1b_2b_3\]
and the constraint $E_{P_{B^*}} f(B^*) = 0$ states that $\alpha_{7}=0$. The identified estimator 
\[\Psi_f(P_n) = \frac{f(0)}{f(0) - \sum_{b \ne 0} f(b) P_n(b)},\] 
where $P_n(b)$ is the observed probability of cell $b$, and the efficient influence curve (equation \ref{efficient_linear}) is  
\[D^*_f = \frac{f(0)}{(f(0) - \sum_{b \ne 0} f(b) P(b))^2} [f(B) - \sum_{b \ne 0} f(b) P(b)].\]

We use the parameters below to generate the underlying distribution. The assumption that the highest-way interaction term $\alpha_7 = 0$ is satisfied in this setting.
\begin{table}[H]
\resizebox{0.65 \textwidth}{!}{
\begin{tabular}{|l|l|l|l|l|l|l|l|}
\hline
$\alpha_0$ & $\alpha_1$ & $\alpha_2$ & $\alpha_3$ & $\alpha_4$ & $\alpha_5$ & $\alpha_6$ & $\alpha_7$ \\ \hline
0.0725 & 0.03 & 0.01 & 0.04 & 0.01 & 0.02 & 0.02 & 0 \\ \hline
\end{tabular}}
\end{table}

The simulated true probabilities for all 7 observed cells are:
\begin{table}[H]
\resizebox{0.8 \textwidth}{!}{
\begin{tabular}{|l|l|l|l|l|l|l|}
\hline
P(0,0,1) & P(0,1,0) & P(0,1,1) & P(1,0,0) & P(1,0,1) & P(1,1,0) & P(1,1,1) \\ \hline
0.1213 & 0.0889 & 0.1429 & 0.1105 & 0.1752 & 0.1428 & 0.2183 \\ \hline
\end{tabular}}
\end{table}

We can calculate the true value of $\Psi_f(P)$ analytically as $\Psi_f(P_0) = 0.9275$, and draw $10^4$ samples to obtain the asymptotic $\Psi_{f,asym}(P_n) = 0.9316$, asymptotic $\hat \sigma^2  = 0.7492$. 

Figure \ref{linear} shows that our estimators perform well when the identification assumption holds true.
\begin{figure}[!ht]
\centering
    \includegraphics[width= 0.8\linewidth]{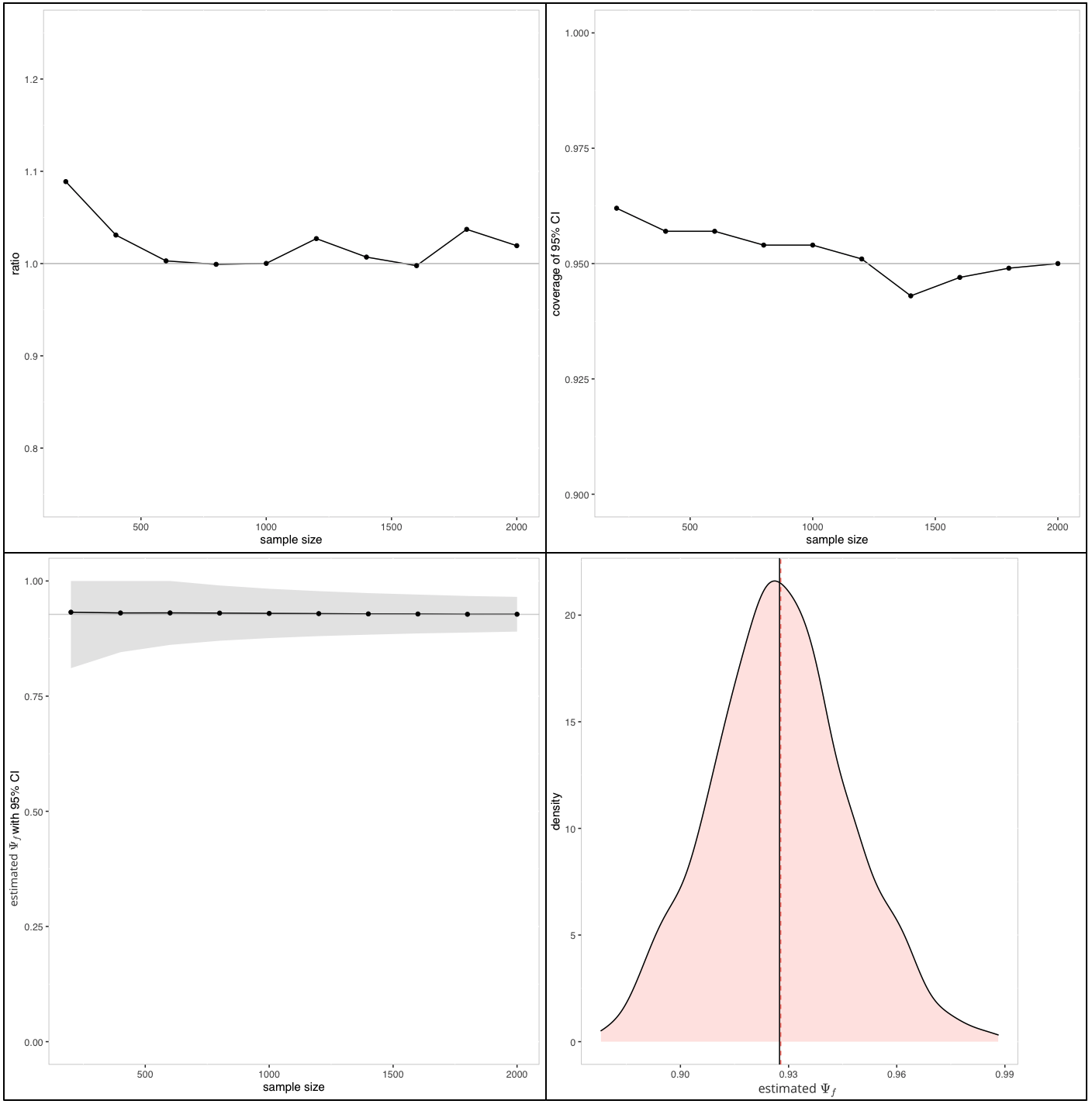}
    \caption{Simulation results when the linear identification assumption holds true.
    \textbf{Upper left}: ratio of estimated variance over true variance vs sample size. The line is the ratio at each sample size, and the grey horizontal line is the true value; 
    \textbf{Upper right}: coverage of 95\% asymptotic confidence interval based on normal approximation with $\hat \sigma^2$ estimated from efficient influence curve vs sample size. The line is the average coverage; 
    \textbf{Lower left}: mean value of estimated $\psi$ and its 95\% asymptotic confidence interval based on normal approximation with $\hat \sigma^2$ estimated from efficient influence curve vs sample size. The line is the mean, and shaded area is the confidence interval and the grey horizontal line is the true value; 
    \textbf{Lower right}: Distribution of 1000 estimates of $\psi$ for sample size of 1000, the vertical line is the true value, and the dashed line is the mean value.}
    \label{linear}
\end{figure}

\newpage

\subsection{Non-linear identification assumption}\label{nonlinear_section}
We perform simulations for three types of non-linear identification assumptions: 1) K-way interaction term equals zero in log-linear models, 2) independence assumption, and 3) conditional independence assumption. For assumption 1), we provide simulations for the following estimators:
\begin{enumerate}
    \item Existing estimator defined by model $M_0$, which assumes that in the log-linear model, all the main terms have the same values, and there are no interaction terms \cite{cormack1989log}: $\Psi_I(P_{M_0})$
    \item Existing estimator defined by model $M_t$, which assumes that the log-linear model does not contain interaction terms \cite{chao2001overview}: $\Psi_I(P_{M_t})$
    \item Non-parametric plug-in estimator: $\Psi_I(P_{NP})$
    \item Estimator based on undersmoothed lasso regression: $\Psi_I(P_{lasso})$
    \item Estimator based on lasso regression with cross-validation: $\Psi_I(P_{lasso\_cv})$
    \item TMLE based on $\Psi_I(P_{lasso})$: $\Psi_I(P_{tmle})$
    \item TMLE based on $\Psi_I(P_{lasso\_cv})$: $\Psi_I(P_{tmle\_cv})$
\end{enumerate}
Model $M_0$ and $M_t$ were calculated using R package "RCapture" (version 1.4-3) \cite{baillargeon2007rcapture}.

\subsection{Non-linear identification assumption: k-way interaction term equals zero} \label{kwayHolds}
Given that we have three samples $B_1, B_2$ and $B_3$, the identification assumption is that in the log-linear model
\[log(E_{B^*}(b)) = \alpha_0 + \alpha_1 b_1 + \alpha_2 b_2 + \alpha_3 b_3 + \alpha_4 b_1b_2 + \alpha_5 b_1b_3 + \alpha_6 b_2b_3 + \alpha_7 b_1b_2b_3,
\]
$\alpha_7 = 0$ (equation \ref{kway_constraint}). Here $E_{B^*}(b)$ is the count of observations in cell $b = (b_1, b_2, b_3)$, and $b_i = 1$ if the subject is captured by sample $i, i = 1,2 ,3$. The parameter $\Psi_I(P)$ is identified in equation \ref{PsiI}, and its influence curve $D^*_{\Phi_I}(P)$ is identified in equation \ref{efficphiI}. 

In this section, we evaluated both our proposed estimators as well as two existing parametric estimators. The first estimator we evaluated is the $M_t$ model \cite{schnabel1938estimation}, where $P_{ij}$ denotes the probability that subject $i$ is captured by sample $j$, and modeled as $log(P_{ij}) = \mu_j$ for subject $i$ and sample $j$. This estimator assumes that all the interaction terms are zero in the log-linear model, and only use the main term variables in model training. The second estimator is $M_{0}$ \cite{cormack1989log}, with the formula $log(P_{ij}) = \alpha$, where $\alpha$ is a constant. 

\subsubsection{Log-linear model with main term effects}\label{case1}

We used the parameters below to generate the underlying distribution. The assumption that the k-way interaction term $\alpha_7=0$ is satisfied in this setting. The model assumptions for $M_0$ and $M_t$ are also satisfied in this setting.

\begin{table}[H]
\resizebox{0.5 \textwidth}{!}{%
\begin{tabular}{|l|l|l|l|l|l|l|l|}
\hline
$\alpha_0$ & $\alpha_1$ & $\alpha_2$ & $\alpha_3$ & $\alpha_4$ & $\alpha_5$ & $\alpha_6$ & $\alpha_7$ \\ \hline
-0.9398 & -1 & -1 & -1 & 0 & 0 & 0 & 0 \\ \hline
\end{tabular}%
}
\end{table}
The simulated observed probabilities for all 7 observed cells are:

\begin{table}[H]
\resizebox{0.8 \textwidth}{!}{%
\begin{tabular}{|l|l|l|l|l|l|l|}
\hline
P(0,0,1) & P(0,1,0) & P(0,1,1) & P(1,0,0) & P(1,0,1) & P(1,1,0) & P(1,1,1) \\ \hline
0.2359 & 0.2359 & 0.0868 & 0.2359 & 0.0868 & 0.0868 & 0.0319 \\ \hline
\end{tabular}%
}
\end{table}

Under this setting, the true value of $\Psi_I(P)$ can be calculated analytically as $\Psi_I(P_0) = 0.6093$. Figure \ref{main_term} and table \ref{main_term_tab} shows the performance of estimators $\Psi_I(P_{M_0}), \Psi_I(P_{M_t}), \Psi_I(P_{NP})$, $\Psi_I(P_{lasso}), \Psi_I(P_{lasso\_cv}), \Psi_I(P_{tmle})$, and $\Psi_I(P_{tmle\_cv})$. Table \ref{MTable0} shows that the model fit statistics for$\Psi_I(P_{M_0})$ and $\Psi_I(P_{M_t})$.

\begin{figure}[H]
\centering
    \includegraphics[width= \linewidth]{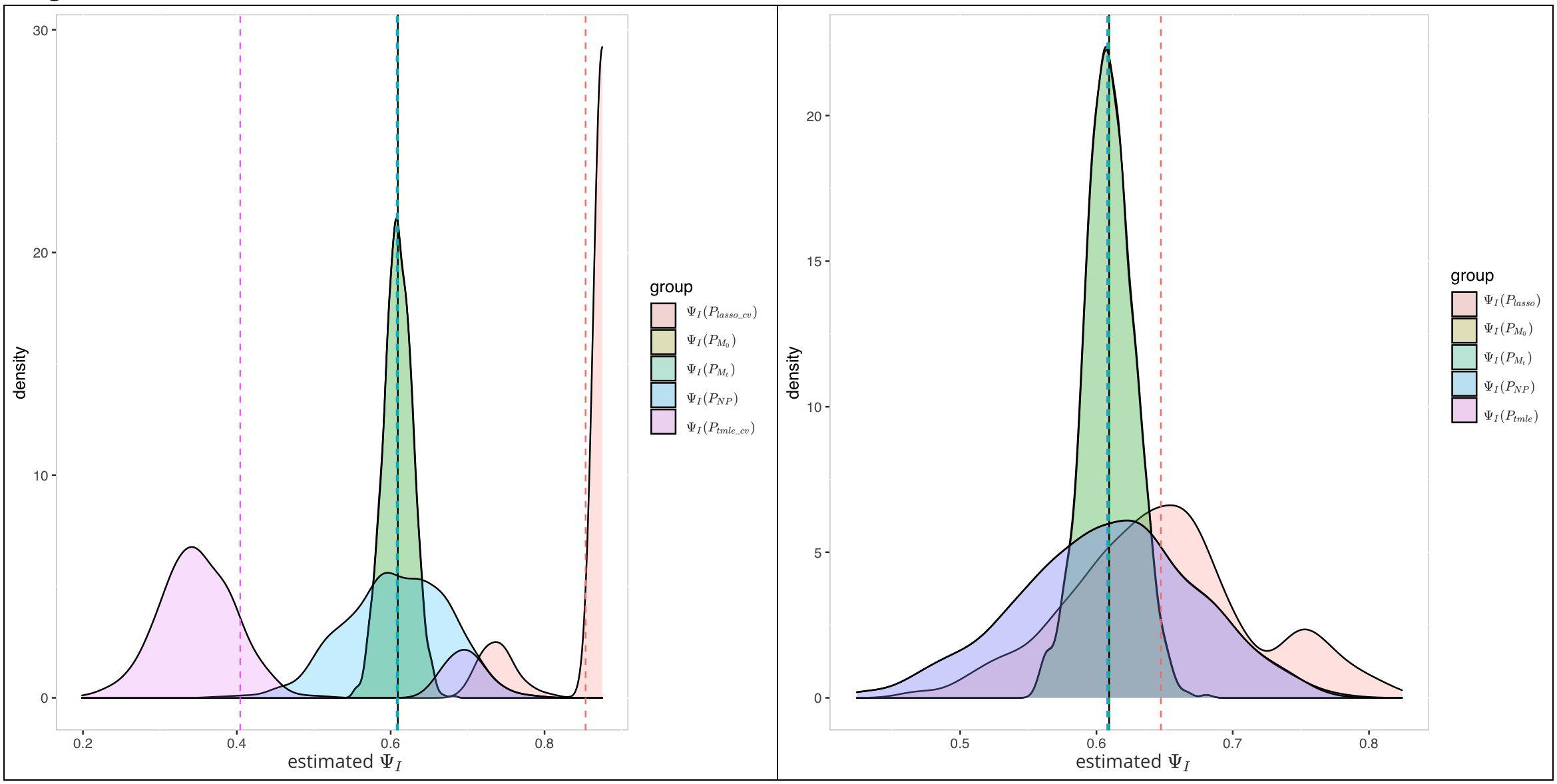}
    \caption{
    \textbf{Left}: Distribution of 1000 estimates of $\psi$ for sample size of 1000 with estimators $\Psi_I(P_{lasso\_cv})$ and $\Psi_I(P_{tmle\_cv})$, the black vertical line is the true value, and the dashed vertical lines represent mean values for each estimator;\textbf{Right}: Distribution of 1000 estimates of $\psi$ for sample size of 1000 with estimators $\Psi_I(P_{lasso})$ and $\Psi_I(P_{tmle})$, the black vertical line is the true value, and the dashed vertical lines represent mean values for each estimator.}
    \label{main_term}
\end{figure}

\begin{table}[H]
\centering
\resizebox{0.8\textwidth}{!}{%
\begin{tabular}{lllll}
\hline
Estimator   & Average $\psi$    & Lower 95\% CI  & Upper 95\% CI & Coverage(\%) \\ \hline
$\Psi_I(P_{NP})$       & 0.6078 & 0.4791 & 0.7366 & 95       \\
$\Psi_I(P_{lasso})$ & 0.6473 & 0.5237 & 0.7708 & 84.2\\
$\Psi_I(P_{tmle})$ & 0.6078 & 0.4806 & 0.735 & 93.7\\
$\Psi_I(P_{lasso\_cv})$ & 0.8534 & 0.7984 & 0.9083 & 4.4      \\
$\Psi_I(P_{tmle\_cv})$ & 0.4044 & 0.2988 & 0.5101 & 13.6     \\
$\Psi_I(P_{M_0})$       & 0.6093 & 0.5659 & 0.6521 & 97.8     \\
$\Psi_I(P_{M_t})$       & 0.6096 & 0.5662 & 0.6525 & 97.8     \\ \hline
\end{tabular}%
}
\caption{Estimated $\hat \psi$ and 95\% asymptotic confidence interval based on normal approximation with $\hat \sigma^2$ estimated from efficient influence curve and its coverage by each estimator. $\Psi_I(P_{NP})$ is the plug-in maximum likelihood estimator, $\Psi_I(P_{lasso})$ uses probabilities estimated from undersmoothed lasso regression, $\Psi_I(P_{tmle})$ is the TMLE based on $\Psi_I(P_{lasso})$, $\Psi_I(P_{lasso\_cv})$ uses probabilities estimated from lasso regression with regularization term optimized by cross-validation, $\Psi_I(P_{tmle\_cv})$ is the TMLE based on $\Psi_I(P_{lasso\_cv})$, $\Psi_I(P_{M_0})$ and $\Psi_I(P_{M_t})$ are existing estimators defined in section \ref{nonlinear_section}. True $\psi_0 = 0.6093$.}
\label{main_term_tab}
\end{table}

\begin{table}[!ht]
\centering
\resizebox{0.4 \textwidth}{!}{%
\begin{tabular}{llll}
\hline
Estimator & Df & AIC & BIC  \\ \hline
$\Psi_I(P_{M_0})$ & 5 & 56.532 & 66.348 \\
$\Psi_I(P_{M_t})$ & 3 & 58.034 & 77.665 \\\hline
\end{tabular}%
}
\caption{model fit statistics for $\Psi_I(P_{M_0})$ and $\Psi_I(P_{M_t})$ estimators.}
\label{MTable0}
\end{table}
Figure \ref{main_term} and table \ref{main_term_tab} compared the performance of two base learners: lasso regression with regularization term chosen by cross-validation ($\Psi_I(P_{lasso\_cv})$), and undersmoothed lasso regression ($\Psi_I(P_{lasso})$).The $\Psi_I(P_{lasso\_cv})$ estimator is significantly biased from the true value of $\psi$, and the TMLE estimator based on it, $\Psi_I(P_{tmle\_cv})$ estimator is also biased. And although the $\Psi_I(P_{lasso})$ estimator is biased, the TMLE based on it, $\Psi_I(P_{tmle})$ estimator ($mean = 0.6078$, $coverage = 93.7\%$) is able to adjust the bias and give a fit as good as the NPMLE estimator ($mean = 0.6078$, $coverage = 95\%$). Thus, in the section below we will only use the undersmooting lasso estimator $\Psi_I(P_{lasso})$ as the base learner.

\newpage
\subsubsection{Log-linear model with main term and interaction term effects}

We used the parameters below to generate the underlying distribution. The assumption that the k-way interaction term $\alpha_7=0$ is satisfied in this setting.

\begin{table}[!ht]
\resizebox{0.6 \textwidth}{!}{%
\begin{tabular}{|l|l|l|l|l|l|l|l|}
\hline
$\alpha_0$ & $\alpha_1$ & $\alpha_2$ & $\alpha_3$ & $\alpha_4$ & $\alpha_5$ & $\alpha_6$ & $\alpha_7$ \\ \hline
-0.9194 & -1 & -1 & -1 & -0.1 & -0.1 & -0.1 & 0 \\ \hline
\end{tabular}%
}
\end{table}
The simulated observed probabilities for all 7 observed cells are:

\begin{table}[!ht]
\resizebox{0.8 \textwidth}{!}{%
\begin{tabular}{|l|l|l|l|l|l|l|}
\hline
P(0,0,1) & P(0,1,0) & P(0,1,1) & P(1,0,0) & P(1,0,1) & P(1,1,0) & P(1,1,1) \\ \hline
0.2440 & 0.2440 & 0.0812 & 0.2440 & 0.0812 & 0.0812 & 0.0245 \\ \hline
\end{tabular}%
}
\end{table}

Under this setting, the true value of $\Psi_I(P)$ can be calculated analytically as $\Psi_I(P_0) = 0.6013$. The model assumptions of $\Psi_I(P_{M_0})$ and $\Psi_I(P_{M_t})$ are violated, and their estimates are biased (estimated mean $\hat \Psi_I(P_{M_0}) = \hat \Psi_I(P_{M_t}) = 0.572, bias = \Psi_I(P_0) - \hat \Psi_I(P_{M_t}) = 0.6013-0.5720 = 0.0293$). The coverage of their 95\% confidence interval is also low (81.5\%). Table \ref{MTable1} shows that the model fit statistics for$\Psi_I(P_{M_0})$ and $\Psi_I(P_{M_t})$. As the $ \Psi_I(P_{NP})$ and $\Psi_I(P_{tmle})$ estimators do not have model assumptions, they are unbiased and the coverage of their 95\% asymptotic confidence intervals are close to 95\%. (94.7\% for $ \Psi_I(P_{NP})$ and $93.7\%$ for $\Psi_I(P_{tmle})$), shown in figure \ref{ls_no_missing} and table \ref{las_tab1}.  

\begin{figure}[H]
\centering
    \includegraphics[width= 0.8\linewidth]{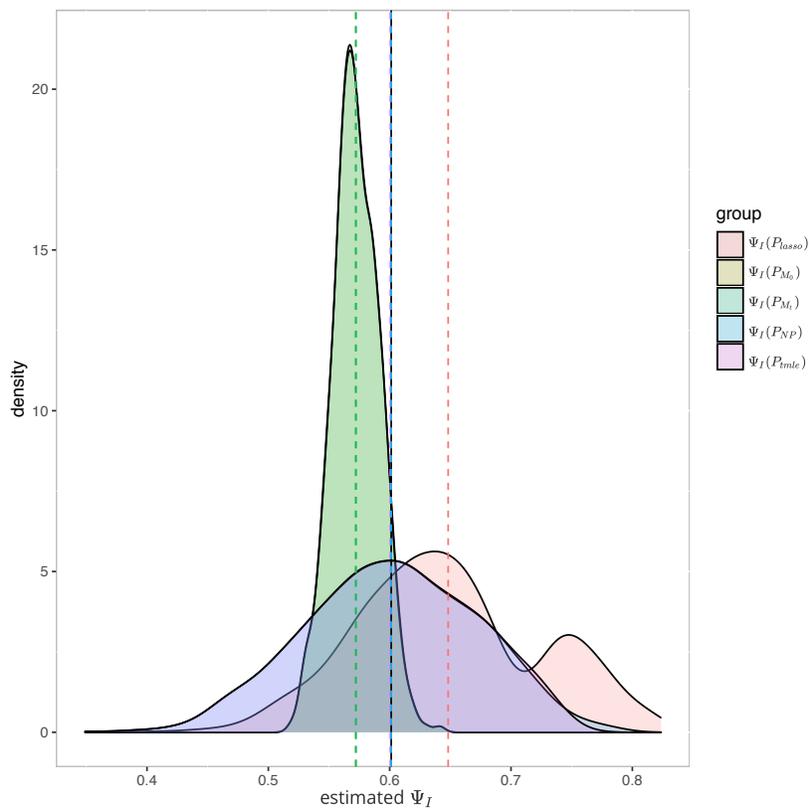}
    \caption{Distribution of 1000 estimates of $\psi$ for sample size of 1000, the black vertical line is the true value, and the dashed vertical lines represent mean values for each estimator.. }
    \label{ls_no_missing}
\end{figure}

\begin{table}[!ht]
\centering
\resizebox{0.8 \textwidth}{!}{%
\begin{tabular}{lllll}
\hline
Estimator   & Average $\psi$    & Lower 95\% CI  & Upper 95\% CI & Coverage(\%) \\ \hline
$\Psi_I(P_{NP})$ & 0.6013 & 0.4625 & 0.7401 & 94.7 \\
$\Psi_I(P_{lasso})$ & 0.6482 & 0.5159 & 0.7805 & 78.4 \\
$\Psi_I(P_{tmle})$ & 0.6005 & 0.4653 & 0.7358 & 93.7 \\
$\Psi_I(P_{M_0})$ & 0.572 & 0.5274 & 0.6163 & 81.5 \\
$\Psi_I(P_{M_t})$ & 0.5722 & 0.5277 & 0.6166 & 81.5 \\\hline
\end{tabular}%
}
\caption{Estimated $\hat \psi$ and 95\% asymptotic confidence interval based on normal approximation with $\hat \sigma^2$ estimated from efficient influence curve and its coverage by each estimator. $\Psi_I(P_{NP})$ is the plug-in maximum likelihood estimator, $\Psi_I(P_{lasso})$ uses probabilities estimated from undersmoothed lasso regression, $\Psi_I(P_{tmle})$ is the TMLE based on $\Psi_I(P_{lasso})$, $\Psi_I(P_{M_0})$ and $\Psi_I(P_{M_t})$ are existing estimators defined in section \ref{nonlinear_section}. True $\Psi_I(P_0) = 0.6013$.}
\label{las_tab1}
\end{table}

\begin{table}[H]
\centering
\resizebox{0.4 \textwidth}{!}{%
\begin{tabular}{llll}
\hline
Estimator & Df & AIC & BIC  \\ \hline
$\Psi_I(P_{M_0})$ & 5 & 55.499 & 65.314 \\
$\Psi_I(P_{M_t})$ & 3 & 58.440 & 78.071 \\ \hline
\end{tabular}%
}
\caption{model fit statistics for$\Psi_I(P_{M_0})$ and $\Psi_I(P_{M_t})$ estimators.}
\label{MTable1}
\end{table}

\newpage
\subsubsection{Log-linear model with empty cells}
When the probability for some cells are close to zero, in the finite sample case, there are likely to be empty cells in the observed data. For example, if the probability of a subject being caught in all 3 samples, represented as $P(1,1,1)$, is less than $10^{-6}$ and the total number of unique subjects caught by any of the 3 samples is less than $10^3$, then it's likely that we will not observe any subject being caught three times, i.e., cell $P(1,1,1)$ will likely be empty.

We used the parameters below to generate the underlying distribution. The assumption that the k-way interaction term $\alpha_7=0$ is satisfied in this setting.

\begin{table}[!ht]
\resizebox{0.55 \textwidth}{!}{%
\begin{tabular}{|l|l|l|l|l|l|l|l|}
\hline
$\alpha_0$ & $\alpha_1$ & $\alpha_2$ & $\alpha_3$ & $\alpha_4$ & $\alpha_5$ & $\alpha_6$ & $\alpha_7$ \\ \hline
-0.4578 & -1 & -2 & -3 & -1 & -1 & -1 & 0 \\ \hline
\end{tabular}%
}
\end{table}

The simulated observed probabilities for all 7 observed cells are:

\begin{table}[!ht]
\resizebox{0.8 \textwidth}{!}{%
\begin{tabular}{|l|l|l|l|l|l|l|}
\hline
P(0,0,1) & P(0,1,0) & P(0,1,1) & P(1,0,0) & P(1,0,1) & P(1,1,0) & P(1,1,1) \\ \hline
0.0857 & 0.2331 & 0.0043 & 0.6336 & 0.0116 & 0.0315 & 2e-04 \\ \hline
\end{tabular}%
}
\end{table}

Figure \ref{missing} and table \ref{las_tab2} show the performance of the estimators when there's no observation in cell $(1,1,1)$, . 

Figure \ref{missing} shows that all the existing estimators are biased when empty cells of $(1,1,1)$ exist. Table \ref{las_tab2} shows that the coverage of the 95\% asymptotic confidence intervals for $\Psi_I(P_{tmle})$ is the highest (58.8\%), and the coverage of the 95\% asymptotic confidence intervals for $\Psi_I(P_{M_0})$, $\Psi_I(P_{M_t})$ are both 0, due to the estimation bias and narrow range of intervals. Table \ref{MTable2} shows the model fit statistics for$\Psi_I(P_{M_0})$ and $\Psi_I(P_{M_t})$. 

\begin{figure}[!ht]
\centering
    \includegraphics[width= 0.8\linewidth]{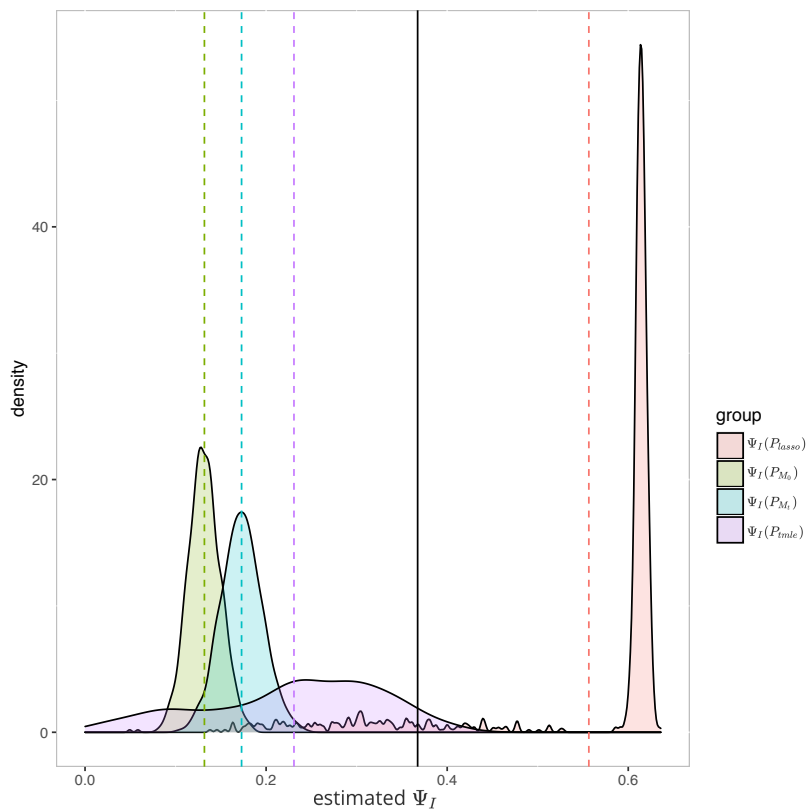}
    \caption{Distribution of 1000 estimates of $\psi$ for sample size of 1000, the black vertical line is the true value of $\Psi_I(P_0) = 0.3674$, and the dashed vertical lines represent mean values for each estimator.}
    \label{missing}
\end{figure}

\begin{table}[!ht]
\centering
\resizebox{0.8 \textwidth}{!}{%
\begin{tabular}{lllll}
\hline
Estimator & Average $\psi$ & Lower 95\% CI & Upper 95\% CI & Coverage(\%)\\ \hline
$\Psi_I(P_{lasso})$ & 0.5565 & 0.1956 & 0.9174 & 20\\
$\Psi_I(P_{tmle})$ & 0.2308 & 0.0875 & 0.374 & 58.8\\
$\Psi_I(P_{M_0})$ & 0.1318 & 0.0995 & 0.169 & 0\\
$\Psi_I(P_{M_t})$ & 0.1729 & 0.1313 & 0.2201 & 0\\\hline
\end{tabular}%
}
\caption{Estimated $\hat \psi$, 95\% asymptotic confidence interval based on normal approximation with $\hat \sigma^2$ estimated from efficient influence curve and its coverage of the  estimators in figure \ref{missing}. $\Psi_I(P_{lasso})$ uses probabilities estimated from undersmoothed lasso regression, $\Psi_I(P_{tmle})$ is the TMLE based on $\Psi_I(P_{lasso})$, $\Psi_I(P_{M_0})$ and $\Psi_I(P_{M_t})$ are existing estimators defined in section \ref{nonlinear_section}.True $\Psi_I(P_0) = 0.3674$.}
\label{las_tab2}
\end{table}

\begin{table}[!ht]
\centering
\resizebox{0.4 \textwidth}{!}{%
\begin{tabular}{llll}
\hline
Estimator & Df & AIC & BIC  \\ \hline
$\Psi_I(P_{M_0})$ & 5 & 501.562 & 511.378 \\
$\Psi_I(P_{M_t})$ & 3 & 45.506 & 65.137 \\ \hline
\end{tabular}%
}
\caption{model fit statistics for$\Psi_I(P_{M_0})$ and $\Psi_I(P_{M_t})$ estimators.}
\label{MTable2}
\end{table}

\subsection{Non-linear identification assumption: independence}\label{indep_section}

Given 3 samples, we assume that the first and second sample $B_1, B_2$ are independent of each other, that is, $P^*(B_1 = 1, B_2 = 1) = P^*(B_1 = 1) \times P^*(B_2 = 1)$. Then we have $\Psi_{II}(P)$ identified in equation \ref{PsiII}, and the influence curve $D^*_{II}(P)$ identified in equation \ref{efficphiII}.

Here we illustrate the performance of our estimators with the following underlying distribution:

\begin{align*}
& P(B_1 = 1) = 0.1\\
& P(B_2 = 1 \vert B_1 = 1) = 0.2\\
& P(B_2 = 1 \vert B_1 = 0) = 0.2\\
& P(B_3 = 1 \vert B_1 = 1) = 0.25\\
& P(B_3 = 1 \vert B_1 = 0) = 0.3\\
\end{align*}

The probability distribution for all 7 observed cells are:

\begin{table}[!ht]
\resizebox{0.8 \textwidth}{!}{%
\begin{tabular}{|l|l|l|l|l|l|l|}
\hline
P(0,0,1) & P(0,1,0) & P(0,1,1) & P(1,0,0) & P(1,0,1) & P(1,1,0) & P(1,1,1) \\ \hline
0.4355 & 0.2540 & 0.1089 & 0.1210 & 0.0403 & 0.0302 & 0.0101\\ \hline
\end{tabular}%
}
\end{table}

We can calculate the true value of $\Psi_{II}(P)$ analytically as $\Psi_{II}(P_0) = 0.4960$, and draw $10^6$ samples to obtain the asymptotic $\Psi_{II}(P_n) = 0.4963$, asymptotic $\sigma^2 = \frac{1}{n} \sum D^{*2} = 4.2657$.

Figure \ref{indep} shows that our estimators perform well when the independence assumptions hold true.
\begin{figure}[H]
\centering
    \includegraphics[width= \linewidth]{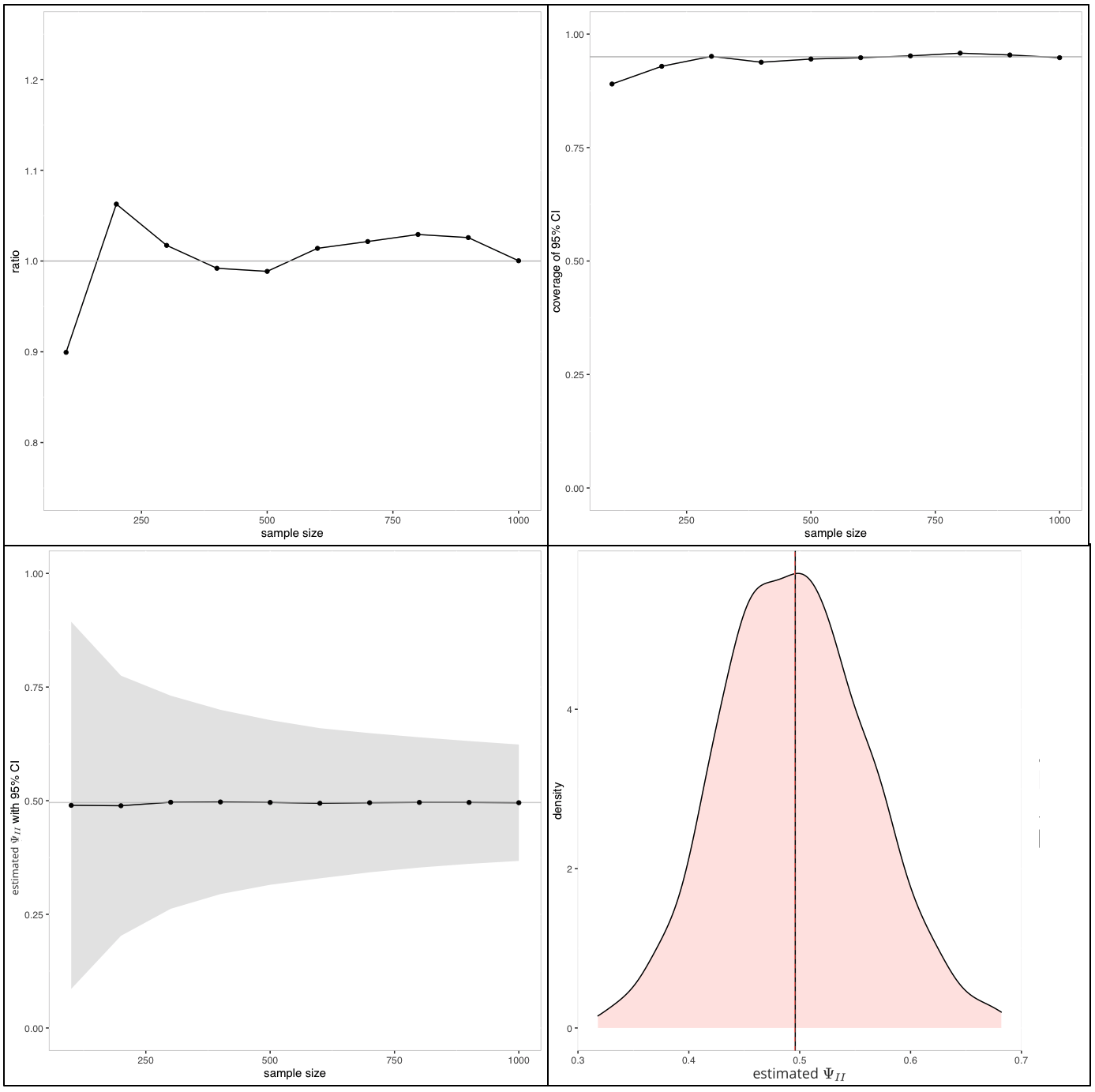}
    \caption{Simulation results when the independence identification assumption holds true.
    \textbf{Upper left}: ratio of estimated variance over true variance vs sample size. The line is the ratio at each sample size, and the grey horizontal line is the true value; 
    \textbf{Upper right}: coverage of 95\% asymptotic confidence interval based on normal approximation with $\hat \sigma^2$ estimated from efficient influence curve vs sample size. The line is the average coverage; 
    \textbf{Lower left}: mean value of estimated $\psi$ and its 95\% asymptotic confidence interval based on normal approximation with $\hat \sigma^2$ estimated from efficient influence curve vs sample size. The line is the mean, and shaded area is the confidence interval and the grey horizontal line is the true value; 
    \textbf{Lower right}: Distribution of 1000 estimates of $\psi$ for sample size of 1000, the vertical line is the true value, and the dashed line is the mean value.}
    \label{indep}
\end{figure}

\newpage

\subsection{Non-linear identification assumption: conditional independence}\label{cond_section}
Given 3 samples, we assume that the third sample $B_3$ is conditionally independent on the second sample $B_2$, that is, $P^*(B_3 = 1 | B_2 = b_2, B_1 = 0) = P^*(B_3 = 1| B_1 = 0)$. Then we can derive that $P^*(0,0,0) = \frac{P^*(0,1,0)P^*(0,0,1)}{P^*(0,1,1)}$ and $\Psi_{CI}(P) = \frac{P(0,1,1)}{P(0,1,1) + P(0,1,0)P(0,0,1)}$ from equation \ref{Psi_CI_Eq}, and the efficient influence curve can be derived from equation \ref{efficphiCI} as
\begin{eqnarray*}
\lefteqn{D^*_{\Phi_{CI}}(P) = } \\
& &  \frac{P(0,1,0)P(0,0,1)}{[P(0,1,1) + P(0,1,0)P(0,0,1)]^2} [I(0,1,1) - P(0,1,1)] - \\
& & \frac{P(0,0,1)P(0,1,1)}{[P(0,1,1) + P(0,1,0)P(0,0,1)]^2} [I(0,1,0) - P(0,1,0)]
- \\
& & \frac{P(0,1,0)P(0,1,1)}{[P(0,1,1) + P(0,1,0)P(0,0,1)]^2} [I(0,0,1) - P(0,0,1)]
\end{eqnarray*}

Here we illustrate the performance of our estimators with the following underlying distribution:

\begin{align*}
& P(B_1 = 1) = 0.1\\
& P(B_2 = 1 \vert B_1 = 1) = 0.2\\
& P(B_2 = 1 \vert B_1 = 0) = 0.15\\
& P(B_3 = 1 \vert B_1 = 1) = P(B_3 = 1 \vert B_2 = 1, B_1 = 1) = 0.25\\
& P(B_3 = 1 \vert B_1 = 1) = P(B_3 = 1 \vert B_2 = 0, B_1 = 1) = 0.25\\
& P(B_3 = 1 \vert B_1 = 0) = P(B_3 = 1 \vert B_2 = 1, B_1 = 0) = 0.2\\
& P(B_3 = 1 \vert B_1 = 0) = P(B_3 = 1 \vert B_2 = 0, B_1 = 0) = 0.2
\end{align*}

The probability distribution for all 7 observed cells $P(B_1, B_2, B_3)$ are:

\begin{table}[H]
\resizebox{0.8 \textwidth}{!}{%
\begin{tabular}{|l|l|l|l|l|l|l|}
\hline
P(0,0,1) & P(0,1,0) & P(0,1,1) & P(1,0,0) & P(1,0,1) & P(1,1,0) & P(1,1,1) \\ \hline
0.3943 & 0.2784 & 0.0696 & 0.1546 & 0.0515 & 0.0387 & 0.0129 \\ \hline
\end{tabular}%
}
\end{table}
We can derive the true $\Psi_{CI}(P)$ analytically as $\Psi_{CI}(P_0) = 0.388$, and draw $10^6$ samples to obtain the asymptotic $\hat \Psi_{CI}(P_n) = 0.3887$, asymptotic $\hat \sigma^2 = 1.0958$.

Figure \ref{ci1} shows that our estimator performs well when the conditional independence assumptions are met.
\begin{figure}[!ht]
\centering
    \includegraphics[width= 0.9\linewidth]{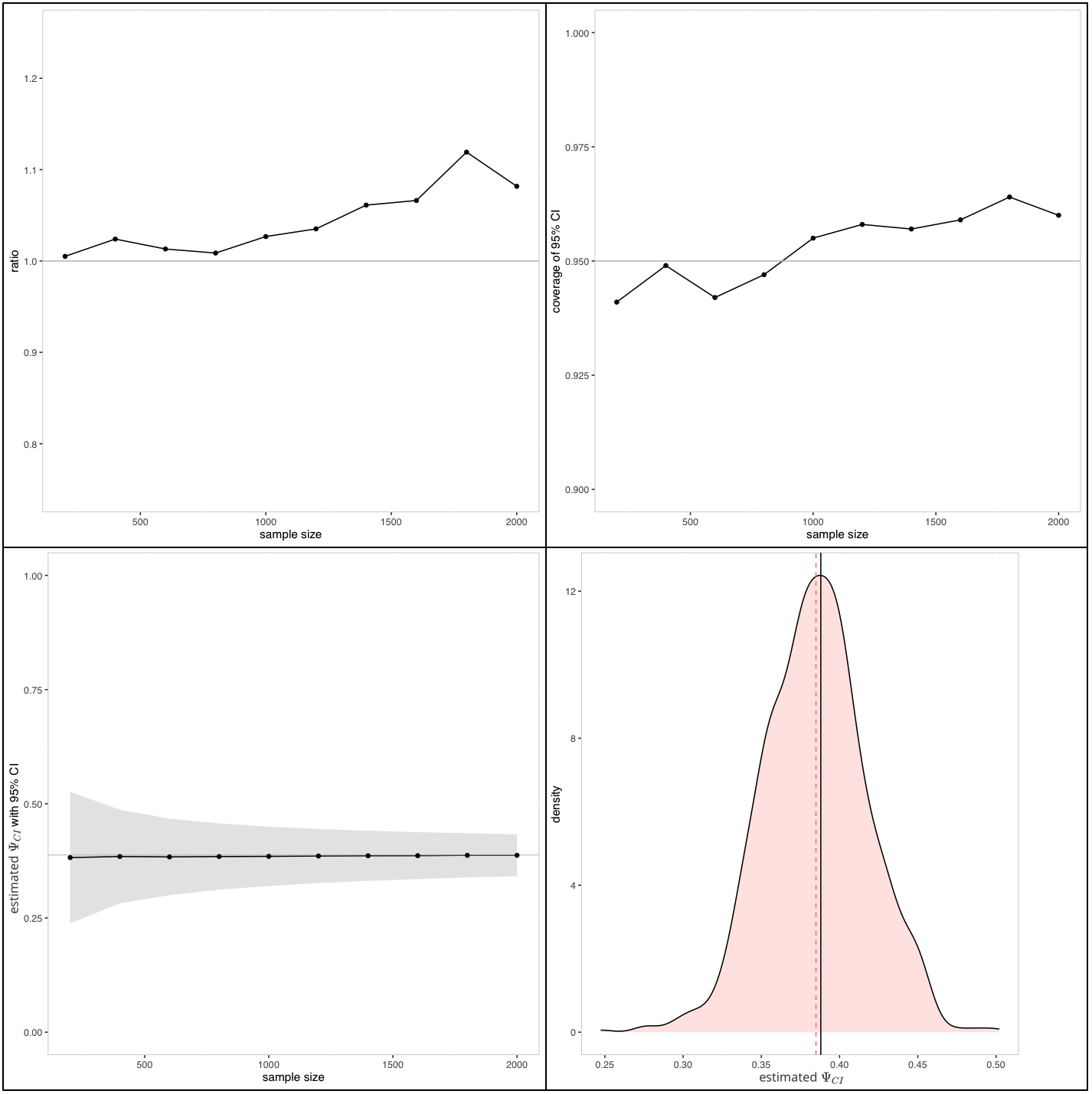}
    \caption{Simulation results when the conditional independence identification assumption holds true.
    \textbf{Upper left}: ratio of estimated variance over true variance vs sample size. The line is the ratio at each sample size, and the grey horizontal line is the true value; 
    \textbf{Upper right}: coverage of 95\% asymptotic confidence interval based on normal approximation with $\hat \sigma^2$ estimated from efficient influence curve vs sample size. The line is the average coverage; 
    \textbf{Lower left}: mean value of estimated $\psi$ and its 95\% asymptotic confidence interval based on normal approximation with $\hat \sigma^2$ estimated from efficient influence curve vs sample size. The line is the mean, and shaded area is the confidence interval and the grey horizontal line is the true value; 
    \textbf{Lower right}: Distribution of 1000 estimates of $\psi$ for sample size of 1000, the vertical line is the true value, and the dashed line is the mean value.}
    \label{ci1}
\end{figure}

\newpage

\clearpage

\section{Evaluating the sensitivity of identification bias to violations of the assumed constraint}\label{sensitivity}
In this section, we use simulations to show the sensitivity of identification bias for all the estimators with identification assumptions violations. 

\subsection{Violation of linear assumptions}
In this section we analyzed the same problem stated in section \ref{linear_section}, but the identification assumption is violated.

We used the parameters below to generate the underlying distribution. The assumption that the k-way interaction term $\alpha_7 = 0$ is violated in this setting($\alpha_7 = 0.2$).

\begin{table}[!ht]
\resizebox{0.6 \textwidth}{!}{%
\begin{tabular}{|l|l|l|l|l|l|l|l|}
\hline
$\alpha_0$ & $\alpha_1$ & $\alpha_2$ & $\alpha_3$ & $\alpha_4$ & $\alpha_5$ & $\alpha_6$ & $\alpha_7$ \\ \hline
0.11 & 0.1 & 0.05 & 0.08 & -0.2 & -0.2 & -0.1 & 0.2 \\ \hline
\end{tabular}%
}
\end{table}

The simulated observed probabilities for all 7 cells are:

\begin{table}[!ht]
\resizebox{0.8 \textwidth}{!}{%
\begin{tabular}{|l|l|l|l|l|l|l|}
\hline
P(0,0,1) & P(0,1,0) & P(0,1,1) & P(1,0,0) & P(1,0,1) & P(1,1,0) & P(1,1,1) \\ \hline
0.2135 & 0.1798 & 0.0449 & 0.2360 & 0.1011 & 0.1978 & 0.0449\\ \hline
\end{tabular}%
}
\end{table}

We can calculate the true value of $\Psi$ analytically as $\Psi_0 = 0.8900$, and draw $10^5$ samples to obtain the asymptotic $\Psi_f(P_n) = 0.7419$, asymptotic $\hat \sigma^2 = 0.2662$. The asymptotic value is biased by 0.1481.

Figure \ref{linear_violated} shows that the linear estimator is biased when the identification assumption is violated.
\begin{figure}[!ht]
\centering
    \includegraphics[width= \linewidth]{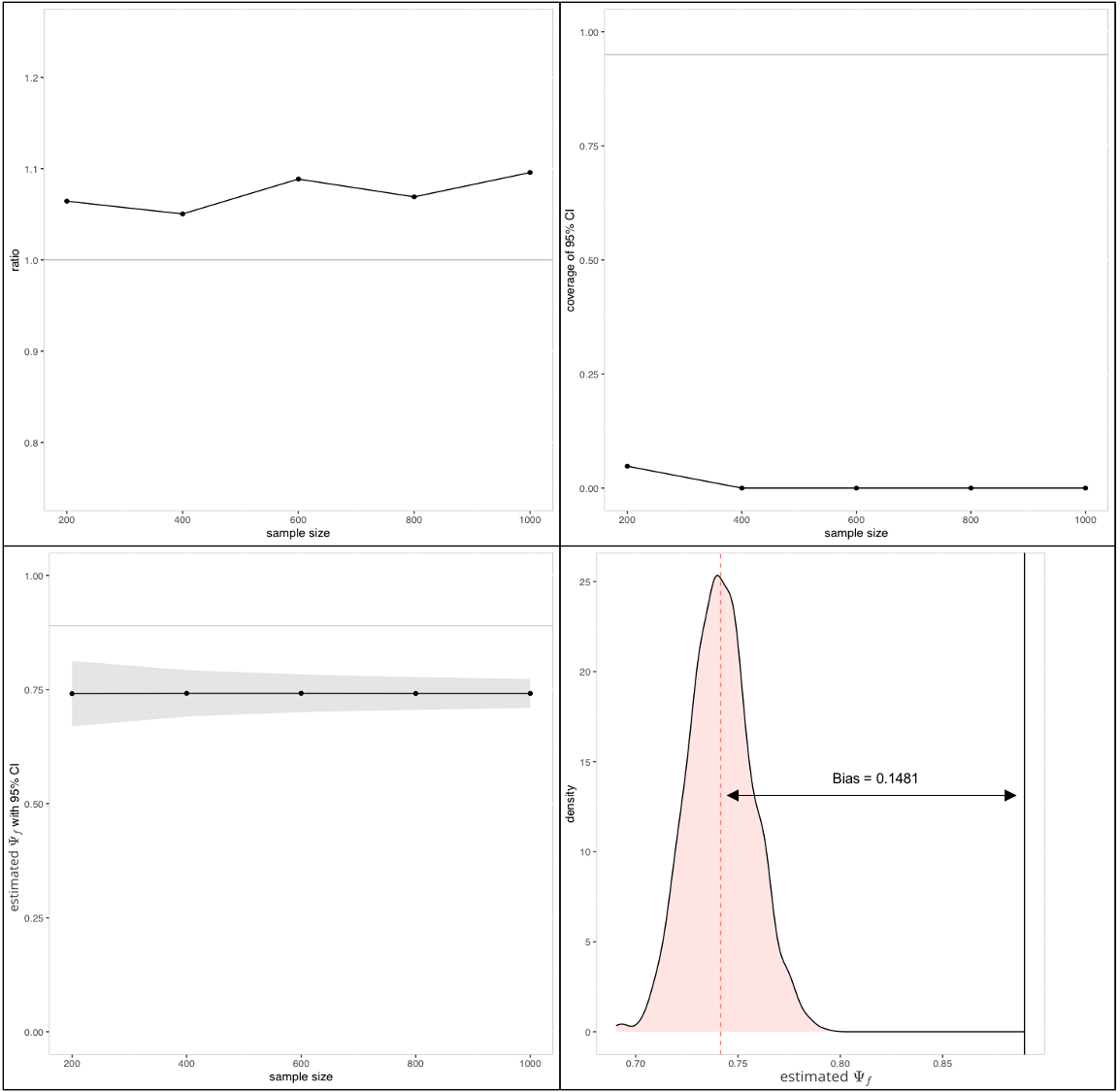}
    \caption{Simulation results when the linear identification assumption is violated.
    \textbf{Upper left}: ratio of estimated variance over true variance vs sample size. The line is the ratio at each sample size, and the grey horizontal line is the true value; 
    \textbf{Upper right}: coverage of 95\% asymptotic confidence interval based on normal approximation with $\hat \sigma^2$ estimated from efficient influence curve vs sample size. The line is the average coverage; 
    \textbf{Lower left}: mean value of estimated $\psi$ and its 95\% asymptotic confidence interval based on normal approximation with $\hat \sigma^2$ estimated from efficient influence curve vs sample size. The line is the mean, and shaded area is the confidence interval and the grey horizontal line is the true value; 
    \textbf{Lower right}: Distribution of 1000 estimates of $\psi$ for sample size of 1000, the vertical line is the true value, and the dashed line is the mean value.}
    \label{linear_violated}
\end{figure}

\newpage

\subsection{Violation of non-linear assumption: independence}

Given 3 samples, we suppose that the first and second sample $B_1, B_2$ are independent of each other, that is, $P^*(B_1 = 1, B_2 = 1) = P^*(B_1 = 1) \times P^*(B_2 = 1)$. Then we can compute $\hat \psi$ and its efficient influence curve as stated above.

Here we illustrate the performance of our estimators with the following underlying distribution:

\begin{align*}
& P(B_1 = 1) = 0.5\\
& P(B_2 = 1 \vert B_1 = 1) = 0.6\\
& P(B_2 = 1 \vert B_1 = 0) = 0.5\\
& P(B_3 = 1) = 0.5
\end{align*}

The probability for all 7 observed cells is given by:

\begin{table}[!ht]
\resizebox{0.8\textwidth}{!}{%
\begin{tabular}{|l|l|l|l|l|l|l|}
\hline
P(0,0,1) & P(0,1,0) & P(0,1,1) & P(1,0,0) & P(1,0,1) & P(1,1,0) & P(1,1,1) \\ \hline
0.1429 & 0.1429 & 0.1429 & 0.1143 & 0.1143 & 0.1714 & 0.1714 \\ \hline
\end{tabular}%
}
\end{table}

We can calculate the true value of $\Psi$ analytically as  $\Psi_0 = 0.875$, and draw $10^6$ samples to obtain the asymptotic $\Psi_{II}(P_n) = 0.9544$, asymptotic $\sigma^2 = 0.4415$, the asymptotic value is biased by 0.0794.

Figure \ref{in1} shows that violation of the independence assumption will create a significant bias in the results.
\begin{figure}[!ht]
\centering
    \includegraphics[width= 0.9\linewidth]{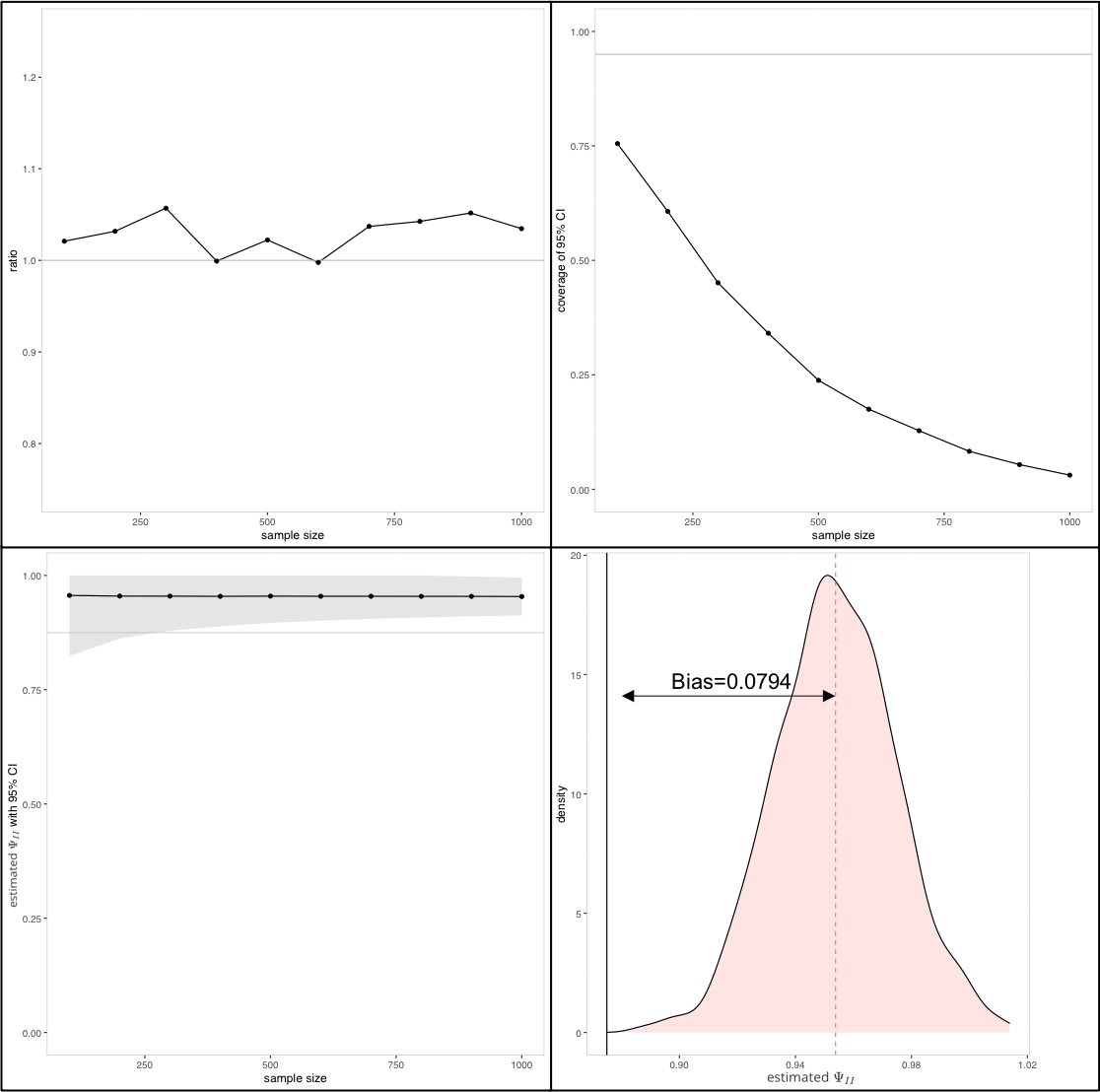}
    \caption{Simulation results when the independence identification assumption is violated.
    \textbf{Upper left}: ratio of estimated variance over true variance vs sample size. The line is the ratio at each sample size, and the grey horizontal line is the true value; 
    \textbf{Upper right}: coverage of 95\% asymptotic confidence interval based on normal approximation with $\hat \sigma^2$ estimated from efficient influence curve vs sample size. The line is the average coverage; 
    \textbf{Lower left}: mean value of estimated $\psi$ and its 95\% asymptotic confidence interval based on normal approximation with $\hat \sigma^2$ estimated from efficient influence curve vs sample size. The line is the mean, and shaded area is the confidence interval and the grey horizontal line is the true value; 
    \textbf{Lower right}: Distribution of 1000 estimates of $\psi$ for sample size of 1000, the vertical line is the true value, and the dashed line is the mean value.}
    \label{in1}
\end{figure}

\clearpage

\subsection{Violation of non-linear assumption: conditional independence}
Same as above, our assumption is that given 3 samples, the third sample $B_3$ is independent on the second sample $B_2$ conditional on the first sample $B_1$. Then we can derive the influence curve same as above. 

Here we illustrate the performance of our estimators when the identification assumption does not hold.

The simulated distribution is as follows:
\begin{align*}
& P(B_1 = 1) = 0.1\\
& P(B_2 = 1 \vert B_1 = 1) = 0.2\\
& P(B_2 = 1 \vert B_1 = 0) = 0.15\\
& P(B_3 = 1 \vert B_1 = 1, B_2 = 1) = 0.10\\
& P(B_3 = 1 \vert B_1 = 1, B_2 = 0) = 0.50\\
& P(B_3 = 1 \vert B_1 = 0, B_2 = 1) = 0.20\\
& P(B_3 = 1 \vert B_1 = 0, B_2 = 0) = 0.50\\
\end{align*}

The probability distribution for all 7 observed cells are:

\begin{table}[!ht]
\resizebox{0.8 \textwidth}{!}{%
\begin{tabular}{|l|l|l|l|l|l|l|}
\hline
P(0,0,1) & P(0,1,0) & P(0,1,1) & P(1,0,0) & P(1,0,1) & P(1,1,0) & P(1,1,1) \\ \hline
0.6194 & 0.1749 & 0.0437 & 0.0648 & 0.0648 & 0.0291 & 0.0032 \\ \hline
\end{tabular}%
}
\end{table}

The true value of $\Psi$ is $\Psi_0 = 0.6175$, and we drew $10^6$ samples and obtained the asymptotic $\Psi_{CI}(P_n) = 0.2931$, asymptotic $\sigma^2 = 1.2349$, the asymptotic value is biased by 0.3244.

Figure \ref{cd} shows that with a strong violation of conditional independence assumptions, the estimated values will have significant biases.

\begin{figure}[!ht]
\centering
    \includegraphics[width= 0.9\linewidth]{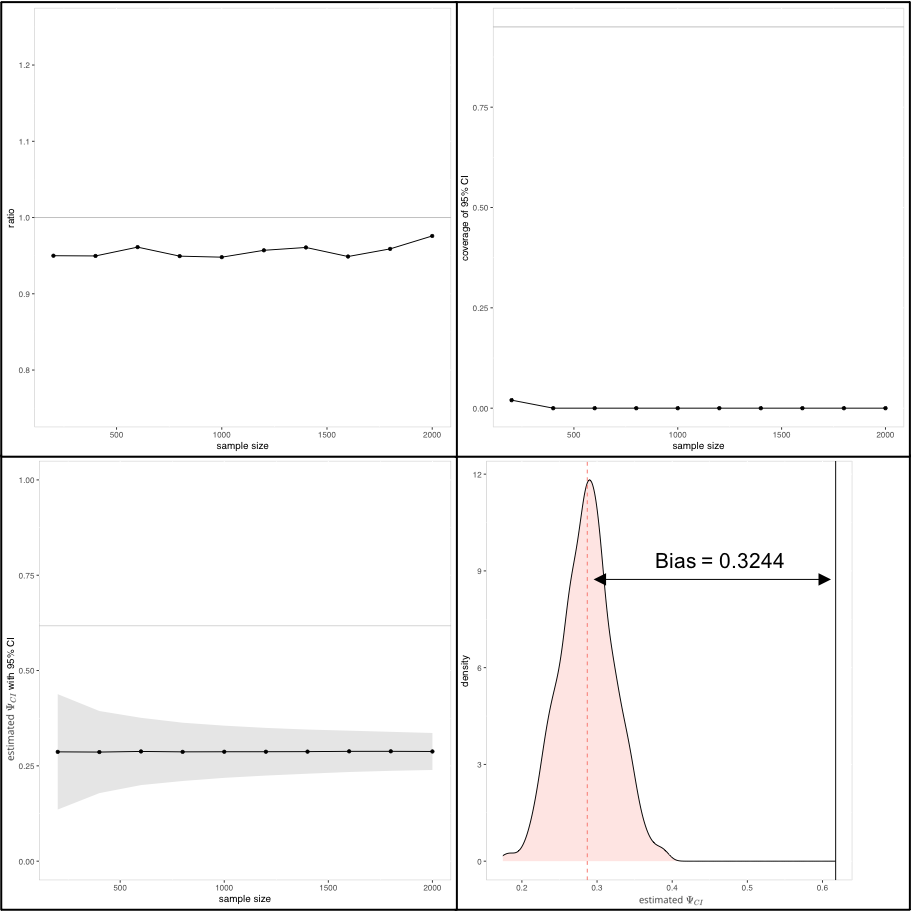}
    \caption{Simulation results when the conditional independence identification assumption is violated.
    \textbf{Upper left}: ratio of estimated variance over true variance vs sample size. The line is the ratio at each sample size, and the grey horizontal line is the true value; 
    \textbf{Upper right}: coverage of 95\% asymptotic confidence interval based on normal approximation with $\hat \sigma^2$ estimated from efficient influence curve vs sample size. The line is the average coverage; 
    \textbf{Lower left}: mean value of estimated $\psi$ and its 95\% asymptotic confidence interval based on normal approximation with $\hat \sigma^2$ estimated from efficient influence curve vs sample size. The line is the mean, and shaded area is the confidence interval and the grey horizontal line is the true value; 
    \textbf{Lower right}: Distribution of 1000 estimates of $\psi$ for sample size of 1000, the vertical line is the true value, and the dashed line is the mean value.}
    \label{cd}
\end{figure}

\newpage

\subsection{Violation of non-linear assumption: K-way interaction term equals zero}
Here we illustrate the performance of our estimators when the identification assumption that the k-way interaction term $\alpha_7=0$ does not hold. We used the parameters below to generate the underlying distribution. The underlying distribution follows the same model as in section \ref{kwayHolds}.

\begin{table}[!ht]
\resizebox{0.55 \textwidth}{!}{
\begin{tabular}{|l|l|l|l|l|l|l|l|}
\hline
$\alpha_0$ & $\alpha_1$ & $\alpha_2$ & $\alpha_3$ & $\alpha_4$ & $\alpha_5$ & $\alpha_6$ & $\alpha_7$ \\ \hline
-1.6333 & 0 & 0 & 0 & -1 & -2 & -0.5 & 1 \\ \hline
\end{tabular}%
}
\end{table}

The simulated observed probabilities for all 7 observed cells are:

\begin{table}[!ht]
\resizebox{0.8\textwidth}{!}{%
\begin{tabular}{|l|l|l|l|l|l|l|}
\hline
P(0,0,1) & P(0,1,0) & P(0,1,1) & P(1,0,0) & P(1,0,1) & P(1,1,0) & P(1,1,1) \\ \hline
0.2386 & 0.2386 & 0.0878 & 0.2386 & 0.0323 & 0.1447 & 0.0196 \\ \hline
\end{tabular}%
}
\end{table}

Figure \ref{strong_violation} and table \ref{strong_violation_table} show that when the assumption is heavily violated, all estimators are significantly biased. The true value of $\psi = 0.8074$. The bias for $ \Psi_I(P_{NP})$ estimator is 0.2032, for $\Psi_I(P_{lasso})$ estimator is 0.1372, for $\Psi_I(P_{tmle})$ is 0.2055, for $\Psi_I(P_{M_0})$ estimator is 0.2217, and for $\Psi_I(P_{M_t})$ estimator is 0.2185. The coverage of 95\% asymptotic confidence intervals for $ \Psi_I(P_{NP})$ estimator is 26\%, for $\Psi_I(P_{lasso})$ estimator is 59\%, for $\Psi_I(P_{tmle})$ is 22\%, for $\Psi_I(P_{M_0})$ and $\Psi_I(P_{M_t})$ is 0\%. Table \ref{MTable4} shows the information on $M_0, M_t$ models.

\begin{figure}[!ht]
\centering
    \includegraphics[width= 0.8 \linewidth]{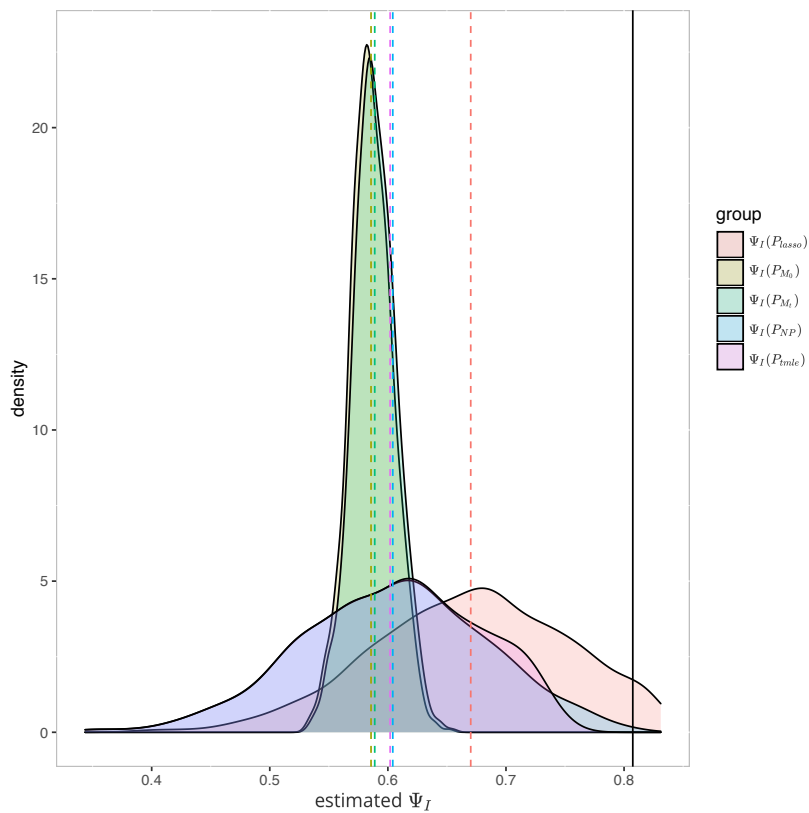}
    \caption{Distribution of 1000 estimates of $\Psi_I$ for sample size of 1000, the black vertical line is the true value, and the dashed vertical lines represent mean values for each estimator.}
    \label{strong_violation}
\end{figure}

\begin{table}[!ht]
\centering
\resizebox{0.8 \textwidth}{!}{%
\begin{tabular}{lllll}
\hline
Estimator   & average $\psi$    & lower 95\% CI  & upper 95\% CI & coverage(\%) \\ \hline
$\Psi_I(P_{NP})$ & 0.6042 & 0.4491 & 0.7592 & 26.3 \\
$\Psi_I(P_{lasso})$ & 0.6702 & 0.5261 & 0.8142 & 58.6 \\
$\Psi_I(P_{tmle})$ & 0.6019 & 0.4516 & 0.7523 & 21.9 \\
$\Psi_I(P_{M_0})$ & 0.5857 & 0.5416 & 0.6295 & 0 \\
$\Psi_I(P_{M_t})$ & 0.5889 & 0.5447 & 0.6328 & 0 \\\hline
\end{tabular}%
}
\caption{Estimated $\hat \psi$ and 95\% asymptotic confidence interval based on normal approximation with $\hat \sigma^2$ estimated from efficient influence curve and its coverage by each estimator. True $\Psi_0 = 0.8074$.}
\label{strong_violation_table}
\end{table}

\begin{table}[!ht]
\centering
\resizebox{0.4 \textwidth}{!}{%
\begin{tabular}{llll}
\hline
Estimator & df & AIC & BIC  \\ \hline
$\Psi_I(P_{M_0})$ & 5 & 157.720 & 167.535 \\
$\Psi_I(P_{M_t})$ & 3 & 128.623 & 148.254  \\ 
\hline
\end{tabular}%
}
\caption{model fit statistics for$\Psi_I(P_{M_0})$ and $\Psi_I(P_{M_t})$ estimators.}
\label{MTable4}
\end{table}

\newpage

\subsection{Violation of all identification assumptions}
In the sections above we showed that if an identification assumption is violated, its corresponding estimators will be biased and their asymptotic confidence intervals will have low coverage. In this section, all the identification assumptions are violated in the simulated distribution, and we analyzed the performance of all the estimators discussed above in this setting.

The underlying distribution of three samples $B_1, B_2, B_3$ follows the same log-linear model in section \ref{kwayHolds}, with the parameters as below.

\begin{table}[!ht]
\resizebox{0.55 \textwidth}{!}{
\begin{tabular}{|l|l|l|l|l|l|l|l|}
\hline
$\alpha_0$ & $\alpha_1$ & $\alpha_2$ & $\alpha_3$ & $\alpha_4$ & $\alpha_5$ & $\alpha_6$ & $\alpha_7$ \\ \hline
-1.1835 & -1 & -1 & -1 & -1.5 & -1 & 2 & 1 \\ \hline
\end{tabular}%
}
\end{table}

The simulated observed probability for all 7 observed cells is given by:

\begin{table}[!ht]
\resizebox{0.8 \textwidth}{!}{%
\begin{tabular}{|l|l|l|l|l|l|l|}
\hline
P(0,0,1) & P(0,1,0) & P(0,1,1) & P(1,0,0) & P(1,0,1) & P(1,1,0) & P(1,1,1) \\ \hline
0.1624 & 0.1624 & 0.0133 & 0.1624 & 0.0220 & 0.4414 & 0.0362 \\ \hline
\end{tabular}%
}
\end{table}

Figure \ref{violated_all_plot} and table \ref{all_violated} show that all estimators are significantly biased, when none of their identification assumption holds. The true value of $\psi = 0.6938$. Under the assumption that sample $B_1$ and $B_2$ are independent conditional on $B_3$ ("Conditional Independence" in table \ref{all_violated}), the plug-in estimator has a bias of -0.2497. Under the assumption that sample $B_1$ and $B_2$ are independent ("Independence" in table \ref{all_violated}), the plug-in estimator has a bias of 0.3760. Under the Under the assumption that the 3-way additive interaction term in the linear model equals zero ("K-way additive interaction equals zero" in table \ref{all_violated}), the plug-in estimator has a bias of -0.2627. And under the assumption that the 3-way multiplicative interaction term in the log-linear model equals zero ("K-way multiplicative interaction equals zero" in table \ref{all_violated}), the $\Psi_I(P_{NP})$ estimator has a bias of 0.2517, the $\Psi_I(P_{lasso})$ estimator has a bias of -0.1573, the $\Psi_I(P_{tmle})$ has a bias of -0.1867, the $\Psi_I(P_{M_0})$ estimator has a bias of -0.1017, and the $\Psi_I(P_{M_t})$ estimator has a bias of -0.1446. The coverage of 95\% asymptotic confidence intervals for all the estimators are far below 95\%. Table \ref{MTable4b} shows the information on $M_0, M_t$ models.

\begin{figure}[!ht]
\centering
    \includegraphics[width= 0.8 \linewidth]{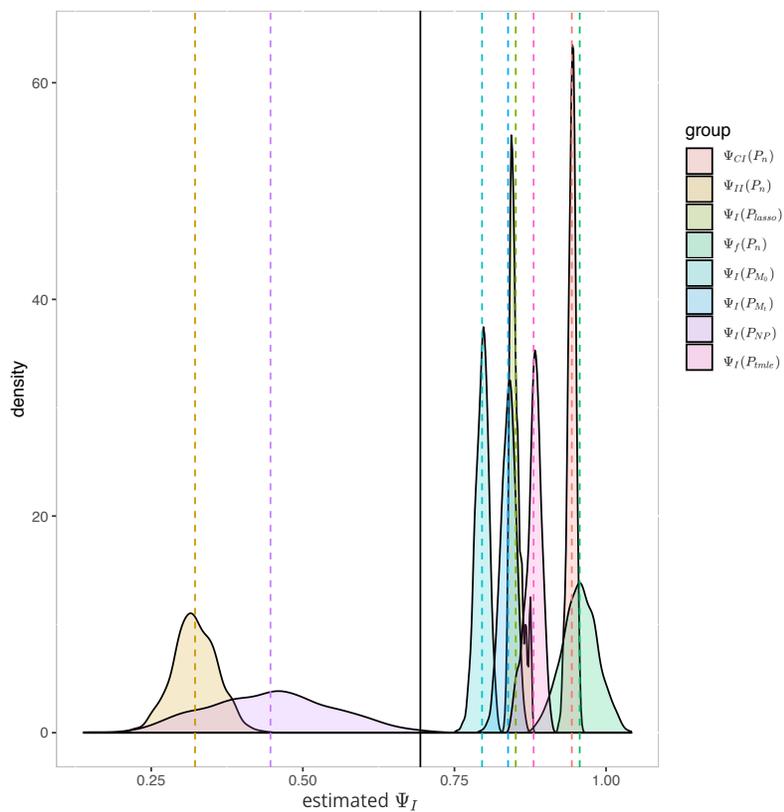}
    \caption{Distribution of 1000 estimates of $\psi$ for sample size of 1000, the black vertical line is the true value, and the dashed vertical lines represent mean values for each estimator.}
    \label{violated_all_plot}
\end{figure}

\begin{landscape}
\begin{table}[!ht]
\centering
\resizebox{1.2\textwidth}{!}{%
\begin{tabular}{llllll}
\hline
Assumption & Estimator  & Average $\psi$& Lower 95\% CI  & Upper 95\% CI & Coverage(\%) \\ \hline

Conditional independence & $\Psi_{CI}(P_n)$ & 0.9435 & 0.9313 & 0.9557 & 31.2\\
Independence & $\Psi_{II}(P_n)$ & 0.3223 & 0.2472 & 0.3974 & 0\\
Highest-way interaction equals zero (linear)& $\Psi_f(P_n)$ & 0.9565 & 0.8999 & 1.0132 & 0\\
Highest-way interaction equals zero (log-linear) & $\Psi_I(P_{NP})$ & 0.4466 & 0.2519 & 0.6413 & 0\\
Highest-way interaction equals zero (log-linear)& $\Psi_I(P_{lasso})$ & 0.8511 & 0.8005 & 0.9017 & 0\\
Highest-way interaction equals zero (log-linear)& $\Psi_I(P_{tmle})$ & 0.8805 & 0.8385 & 0.9225 & 31.2\\
Highest-way interaction equals zero (log-linear)& $\Psi_I(P_{M_0})$ & 0.7955 & 0.7621 & 0.8269 & 0\\
Highest-way interaction equals zero (log-linear)& $\Psi_I(P_{M_t})$ & 0.8384 & 0.8074 & 0.8673 & 0\\\hline
\end{tabular}%
}
\caption{Estimated $\hat \psi$ and 95\% asymptotic confidence interval based on normal approximation with $\hat \sigma^2$ estimated from efficient influence curve and its coverage for each estimator. True $\psi_0 = 0.6938$.
Conditional independence assumption: sample $B_1$ and $B_2$ are independent conditional on $B_3$; Independence assumption: sample $B_1$ and $B_2$ are independent; K-way additive interaction equals zero assumption: 3-way interaction term in linear model equals zero; K-way multiplicative interaction equals zero assumption: 3-way interaction term in log-linear model equals zero.
}
\label{all_violated}
\end{table}

\begin{table}[!ht]
\centering
\resizebox{0.4 \textwidth}{!}{%
\begin{tabular}{llll}
\hline
Estimator & Df & AIC & BIC  \\ \hline
$\Psi_I(P_{M_0})$ & 5 & 743.631 & 753.447 \\
$\Psi_I(P_{M_t})$ & 3 & 391.757 & 411.388  \\ 
\hline
\end{tabular}%
}
\caption{model fit statistics for$\Psi_I(P_{M_0})$ and $\Psi_I(P_{M_t})$ estimators.}
\label{MTable4b}
\end{table}
\end{landscape}
\clearpage

\section{Data Analysis}\label{data}
Schistosomiasis is an acute and chronic parasitic disease endemic to 78 countries worldwide. Estimates show that at least 229 million people required preventive treatment in 2018 \cite{who}. In China, schistosomiasis is an ongoing public health challenge and the country has set ambitious goals of achieving nationwide transmission interruption \cite{liang2014surveillance}. To monitor the transmission of schistosomiasis, China operates three surveillance systems which can be considered as repeated samples of the underlying infected population. The first is a national surveillance system designed to monitor nationwide schistosomiasis prevalence ($S_1$), which covers 1\% of communities in endemic provinces, conducting surveys every 6-9 years. The second is a sentinel system designed to provide longitudinal measures of disease prevalence and intensity in select communities ($S_2$), and conducts yearly surveys. The third system ($S_3$) comprises routine surveillance of all communities in endemic counties, with surveys conducted on a roughly three-year cycle \cite{liang2014surveillance}. In this application, we use case records for schistosomiasis from these three systems for a region in southwestern China \cite{spear2004factors}. The community, having a population of about 6000, reported a total of 302 cases in 2004: 112 in $S_1$; 294 in $S_2$; and 167 in $S_3$; with 202 cases appearing on more than one system register (figure \ref{schisto}). The three surveillance systems are not independent, and we apply estimators under the four constraints (highest-way interaction equals zero in log-linear model, highest-way interaction equals zero in linear model, independence, and conditional independence)to the data.

\begin{figure}[!ht]
\centering
    \includegraphics[width= 0.6 \linewidth]{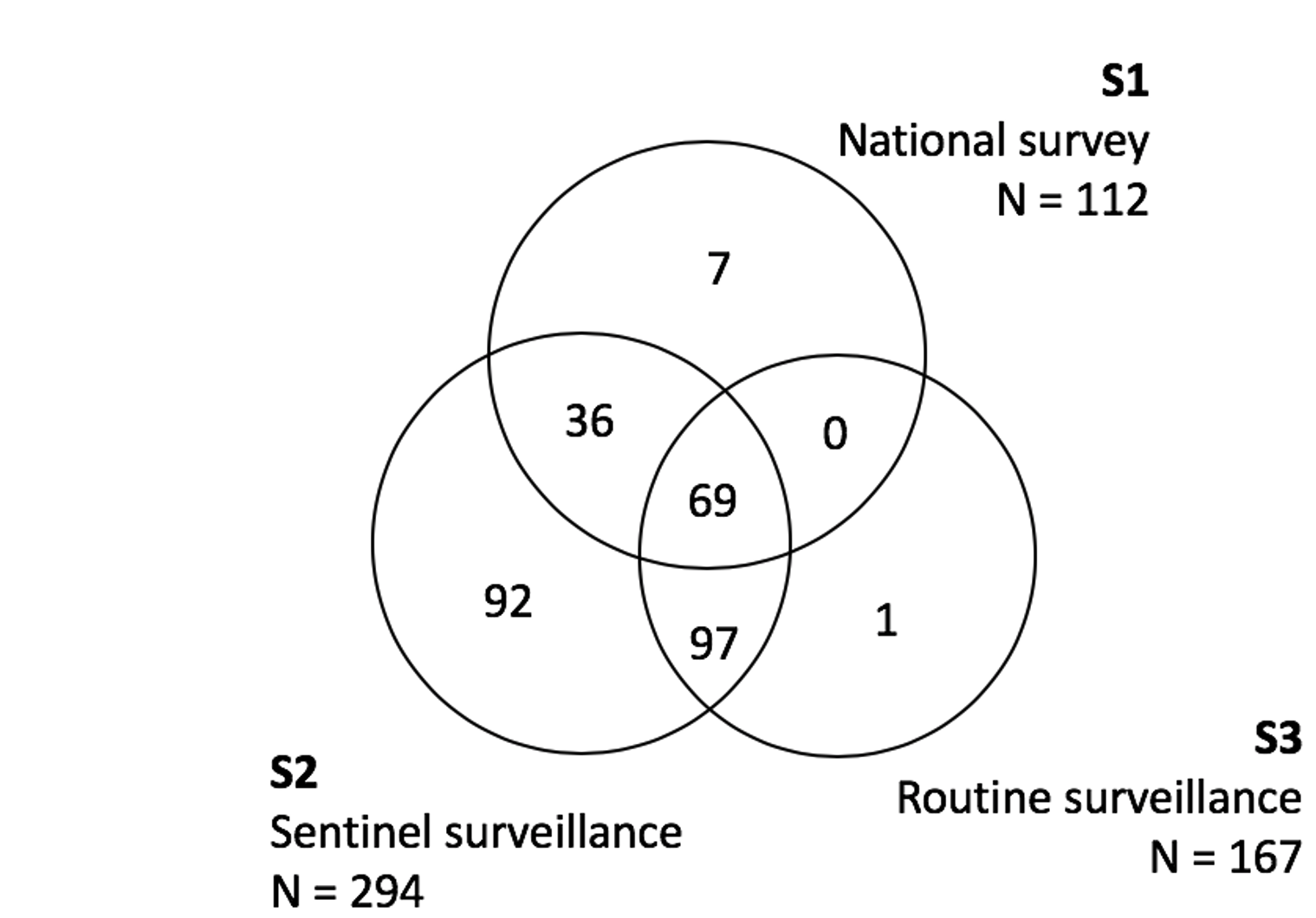}
    \caption{Schistosomiasis case frequencies among $S_1, S_2, S_3$ surveillance systems in a community (population $\sim 6000$) in southwestern China in 2004 \cite{spear2004factors}}
    \label{schisto}
\end{figure}

The estimated $\hat \psi$ with its 95\% asymptotic confidence intervals are shown in table \ref{schisto_table}. The highest estimate of $\psi$ is 0.9969 (with 95\% CI: (0.8041, 0.9829)) under conditional independence assumption, and the lowest is 0.8935 (with 95\% CI: (0.9965, 0.9972)) under the linear model with no highest-way interaction assumption (table \ref{schisto_table}). This analysis illustrates the heavy dependence of estimates of $\Psi$ on the arbitrary identification assumptions using real data.

One way to solve this problem is to hold certain identification assumptions true by design. In this case, we could make an effort to design the sentinel system $S_2$ and routine system $S_3$ to be independent of each other conditional on national system $S_1$. For example, let the sampling surveys in $S_2$ and $S_3$ be done independently, while each of them could borrow information from $S_1$. In so doing we can assume the conditional independence between the two samples and conduct the analysis correspondingly.

\begin{landscape}
\begin{table}[!ht]
\centering
\resizebox{1.2\textwidth}{!}{%
\begin{tabular}{lllll}
\hline
Assumption & Estimator  & Estimated $\psi$& Lower 95\% CI  & Upper 95\% CI\\ \hline
Highest-way interaction equals zero (log-linear) & $\Psi_I(P_{lasso})$ & 0.9212 & 0.8755 & 0.9669\\
Highest-way interaction equals zero (log-linear)& $\Psi_I(P_{tmle})$ & 0.9428 & 0.9089 & 0.9768\\
Highest-way interaction equals zero (log-linear)& $\Psi_I(P_{M_0})$ & 0.9321 & 0.8959 & 0.9627\\
Highest-way interaction equals zero (log-linear)& $\Psi_I(P_{M_t})$ & 0.9934 & 0.9748 & 1\\
Highest-way interaction equals zero (linear)& $\Psi_f(P_n)$ & 0.8935 & 0.8041 & 0.9829\\
Independence & $\Psi_{II}(P_n)$ & 0.963 & 0.9326 & 0.9935\\
Conditional independence &  $\Psi_{CI}(P_n)$ & 0.9969 & 0.9965 & 0.9972\\\hline
\end{tabular}%
}
\caption{Estimated $\hat \psi$ and 95\% asymptotic confidence interval based on normal approximation with $\hat \sigma^2$ estimated from efficient influence curve and its coverage for each estimator. Conditional independence assumption: survey $S_1$ and $S_2$ are independent conditional on $S_3$; Independence assumption: survey $S_1$ and $S_2$ are independent; K-way additive interaction equals zero assumption: 3-way interaction term in linear model equals zero; K-way multiplicative interaction equals zero assumption: 3-way interaction term in log-linear model equals zero. The $\Psi_I(P_{NP})$ estimator under K-way multiplicative interaction term equals zero assumption is not defined due to existing empty cells in observed data.}
\label{schisto_table}
\end{table}

\begin{table}[!ht]
\centering
\resizebox{0.4 \textwidth}{!}{%
\begin{tabular}{llll}
\hline
Estimator & Df & AIC & BIC  \\ \hline
$\Psi_I(P_{M_0})$ & 5 & 329.383 & 336.804 \\
$\Psi_I(P_{M_t})$ & 3 & 60.139 & 74.980  \\ \hline
\end{tabular}%
}
\caption{model fit statistics for$\Psi_I(P_{M_0})$ and $\Psi_I(P_{M_t})$ estimators.}
\label{MTable5}
\end{table}
\end{landscape}
\clearpage

\section{Discussion}\label{conclusions}

\subsection{Summary table}
In this section we summarise all the estimators we proposed under linear and non-linear constraints as below.

\begin{table}[!ht]
\resizebox{\textwidth}{!}{%
\begin{tabular}{@{}ll@{}}
\toprule
Identification   assumption & Highest-way interaction term in linear model equals zero             \\ \midrule
Constraint                  & $E_{P_{B^*}} f(B^*) =\sum_b (-1)^{K+\sum_{k=1}^K b_k} P_{B^*,0}(b)= 0$\\\midrule
Method                      & Plug-in                                                                \\\midrule
Target parameter            & $\psi_f = \frac{f(0)}{f(0) - \sum_{b \ne 0} f(b) P(b)}$                \\\midrule
Efficient influence curve & $D^*_f = \frac{f(0)}{(f(0) - \sum_{b \ne 0} f(b) P(b))^2} [f(B) - \sum_{b   \ne 0} f(b) P(b)]$ \\ \bottomrule
\end{tabular}
}
\caption{Summary table for identification assumption: highest-way interaction term in linear model equals zero, where $f(b)=(-1)^{K+\sum_{k=1}^K b_k}$.}
\end{table}

\begin{table}[!ht]
\resizebox{\textwidth}{!}{%
\begin{tabular}{@{}ll@{}}
\toprule
Identification   assumption & Independence between two samples \\ \midrule
Constraint                  & $\Phi_{II}(P^*)\equiv \sum_b   I(b(1:2)=(0,0))P^*(b) -$\\ 
 & $\sum_{b_1,b_2}I(b_1(1)=b_2(2)=0)P^*(b_1)P^*(b_2)$   

\\\midrule
Method                      & Plug-in         \\\midrule
Target parameter            & $\Psi_{II}(P)=\frac{   1-P(B(1)=0)-P(B(2)=0)+P(B(1:2)=(0,0))}{1-P(B(1)=0)-P(B(2)=0)+P(B(1)=0)P(B(2)=0)}$ \\\midrule
Efficient influence curve &
$D^*_{\Phi_{II}}(P) = C_2(P)\{\mathbbm{1}(B(1)=0)-P(B(1)=0)\} +$\\ 
 & $C_3(P)\{\mathbbm{1}(B(2)=0)-P(B(2)=0) +$\\ 
 & $C_4(P)\{\mathbbm{1}(B(1)=B(2)=0)-P(B(1:2)=0\}$,\\
 & where 
$C_2(P) = \frac{1-P(B(1)=0)-P(B(2)=0)-P(B(1)=B(2)=0)}{(1-P(B(1)=0)-P(B(2)=0)-P(B(1)=0)P(B(2)=0))^2} P(B(2)=0)$,\\
 & $C_3(P) = \frac{1-P(B(1)=0)-P(B(2)=0)-P(B(1)=B(2)=0)}{(1-P(B(1)=0)-P(B(2)=0)-P(B(1)=0)P(B(2)=0))^2} P(B(1)=0)$,\\
 & $C_4(P) = -\frac{1}{1-P(B(1)=0)-P(B(2)=0)-P(B(1)=0)P(B(2)=0)}$
\\ \bottomrule
\end{tabular}}
\caption{Summary table for identification assumption: independence between two samples.}
\end{table}

\begin{table}[!ht]
\resizebox{\textwidth}{!}{%
\begin{tabular}{@{}ll@{}}
\toprule
Identification   assumption & Conditional independence between two   samples \\ \midrule
Constraint & $\Phi_{CI,(j,m)} =P^*(B_j =1   \vert B_1 = 0, \cdots, B_m = 0, \cdots, B_K = 0)$ \\
 & $- P^*(B_j =1 \vert B_1 = 0,   \cdots, B_m = 1, \cdots, B_K = 0)$ \\ \midrule
Method & Plug-in \\ \midrule
Target parameter & $\Psi_{CI} = \frac{P(B_m = 1, B_j = 1, 0,...,0)}{P(B_m = 1, B_j = 1,  0,...,0) + P(B_m = 0, B_j = 1,  0,...,0)P(B_m = 1, B_j = 0,  0,...,0)}$ \\ \midrule
Efficient influence curve & $D^*_{\Phi_{CI}}(P) = \frac{1}{C_5(P)}(C_6(P) - C_7(P) - C_8(P))$,\\
 & where $C_5(P) = -\sum_{b_1\not =0,b_2\not =0}[\mathbbm{I}(b_1(1)=b_2(2)=0)P(b_1)P(b_2)$\\
& $C_6(P) = P(B_m = 0, B_j = 1, 0,...,0)P(B_m = 0, B_j = 0, 0,...,0)$\\
& $[\mathbbm{I}(B_m = 1, B_j = 1, 0,...,0) - P(B_m = 1, B_j = 1, 0,...,0)]$\\
& $C_7(P) = P(B_m = 1, B_j = 0, 0,...,0)P(B_m = 1, B_j = 1, 0,...,0)$\\
& $[\mathbbm{I}(B_m = 0, B_j = 1, 0,...,0) - P(B_m = 0, B_j = 1, 0,...,0)]$\\
& $C_8(P) = P(B_m = 0, B_j = 1, 0,...,0)P(B_m = 1, B_j = 1, 0,...,0)$\\
& $[\mathbbm{I}(B_m = 1, B_j = 0, 0,...,0) - P(B_m = 1, B_j = 0, 0,...,0)]$\\ 

\bottomrule
\end{tabular}}
\caption{Summary table for identification assumption: conditional independence between two samples.}
\end{table}

\begin{table}[!ht]
\resizebox{\textwidth}{!}{%
\begin{tabular}{@{}ll@{}}
\toprule
Identification   assumption & Highest-way interaction term in log-linear model equals zero           \\ \midrule
Constraint                  & $\Phi_I(P^*)\equiv \sum_b (-1)^{1+\sum_{k=1}^K b_k} \log P_{B^*}(b) = 0$ \\ \midrule
Method                      & Plug-in (NPMLE), undersmoothed lasso, TMLE based on lasso               \\ \midrule
Target parameter          & $\Psi_I(P)=\frac{1}{1+\exp( (-1)^{K+1}\sum_{b\not =0}f(b)\log P(b))}$      \\ \midrule
Efficient influence curve & $D^*_{\Phi_I}(P) = (-1)^K \Psi_I(P)(1-\Psi_I(P)) \left\{\frac{f(B)}{P(B)}+f(0) \right\}$ \\ \bottomrule
\end{tabular}%
}
\caption{Summary table for identification assumption: highest-way interaction term in log-linear model equals zero, where $f(b)=(-1)^{K+\sum_{k=1}^K b_k}$.}
\end{table}

\clearpage

We developed a modern method to estimate population size based on capture-recapture designs with a minimal number of constraints or parametric assumptions. We provide the solutions, theoretical support, simulation study and sensitivity analysis for four identification assumptions: independence between two samples, conditional independence between two samples, no highest-way interaction in linear models, and no highest-way  interaction in log-linear models. We also developed machine learning algorithms to solve the curse of dimensionality for high dimensional problems under the assumption of no highest-way interaction in log-linear model. Through our analysis, we found that whether the identification assumption holds true plays a vital role in the performance of estimation. When the assumption is violated, all estimators will be biased. This conclusion applies to models of all forms, parametric or non-parametric, simple plug-in estimators or complex machine-learning based estimators. Thus one should always ensure that the chosen identification assumption is known to be true by survey design, otherwise all the estimators will be unreliable. Under the circumstances where the identification assumptions hold true, the performance of our targeted maximum likelihood estimator, $\Psi_I(P_{tmle})$, is superior to $\Psi_I(P_{M_0})$(identical capture-probabilities, no highest-way interaction in log-linear model), $\Psi_I(P_{M_t})$(no interaction terms in log-linear model) and $\Psi_I(P_{NP})$(plugged-in, no highest-way interaction in log-linear model) estimators in several aspects: first, by making the least number of assumptions required for identifiability, the estimator $\Psi_I(P_{tmle})$ is more robust in empirical data analysis, as there will be no bias due to violations of parametric model assumptions. Second, the estimator $\Psi_I(P_{tmle})$ is based on a consistent undersmoothed lasso estimator. This property ensures the asymptotic efficiency of TMLE. Third, when there are empty cells, the estimator $\Psi_I(P_{tmle})$ solves the curse of dimensionality by correcting the bias introduced by the undersmoothed lasso estimator, and gives a more honest asymptotic confidence interval, wider that those from parametric models, and hence a higher coverage.
\newpage

\printbibliography

\newpage

\section{Appendix}\label{appendix}
In the appendix, we formally state the lemmas used in the context and provide the proofs. In section \ref{apd_indep_psi}, we derive the target parameter $\Psi_{II}(P)$ under identification assumption that the first two samples $B_1, B_2$ are independent, given there are three samples in total. In section \ref{apd_indep_d} we derive the efficient influence curve $D_{\Phi_{II}}^*(P)$ for $\Psi_{II}(P)$. In section \ref{apd_lemma}, we state the lemma on how to derive the efficient influence curve under multidimensional constraint $\Phi_{III}$. In section \ref{apd_condInf_psi}, we formally state the target parameter $\Psi_{CI}(P)$ under identification assumption that the first two samples $B_1, B_2$ are independent conditional on the third samples. In section \ref{apd_condInf_d} we derive the efficient influence curve $D_{\Phi_{CI}}^*(P)$ for $\Psi_{CI}(P)$. In section \ref{apd_efficincy_TMLE} we prove the asymptotic efficiency of the TMLE.

\subsection{Target parameter under independence assumption}\label{apd_indep_psi}
\begin{lemma}
For constraint $\Phi_{II}$(equation \ref{PhiII}), we have the target parameter $\Psi_{II}$ as:
\begin{eqnarray*}
\Psi_{II}(P) = \frac{1 - P(B(1) = 0) - P(B(2) = 0) + P(B(1:2) = (0,0))}{1 - P(B(1) = 0) - P(B(2) = 0) + P(B(1) = 0)P(B(2) = 0)}
\end{eqnarray*}\label{lm_indep_psi}
\end{lemma}

\begin{proof}
We provide a brief proof of lemma \ref{lm_indep_psi} when there are three samples. In this case, $\Phi_{II}(P^*) = 0 $ is equivalent to 
\begin{equation}
P^*(B^*(1:2) = (0,0)) = P^*(B(1) = 0)P^*(B^*(2) = 0).\label{phiII_eq1} 
\end{equation}
When there are three samples, equation \ref{phiII_eq1} can be expanded as:
\begin{eqnarray}
P^*(0,0,1) + P^*(0,0,0) &=& P^*(0,0,0)^2 + [P^*(0,1,0) + P^*(0,1,1)+ P^*(0,0,1) \nonumber\\ 
&& + P^*(1,0,0) + P^*(1,0,1) + P^*(0,0,1)]\times P^*(0,0,0) \nonumber\\
&& + [P^*(0,1,0) + P^*(0,1,1) + P^*(0,0,1)] \nonumber\\
&& \times [P^*(1,0,0) + P^*(1,0,1) + P^*(0,0,1)]\label{eq1}
\end{eqnarray}
Denote $a^* = P^*(0,1,0) + P^*(0,1,1) + P^*(0,0,1)$, $b^* = P^*(1,0,0) + P^*(1,0,1) + P^*(0,0,1)$. Equation \ref{eq1} can be written as:
\begin{eqnarray*}
&0 = P^*(0,0,0)^2 + (a^* + b^* -1)\times P^*(0,0,0) - P^*(0,0,1).
\end{eqnarray*}
Plug in $\psi = 1 - P^*(0,0,0)$, we have
\[P(0,1,0) = \frac{P^*(0,1,0)}{\psi}, ..., P(0,0,1) = \frac{P^*(0,0,1)}{\psi}, a = \frac{a^*}{\psi}, b = \frac{b^*}{\psi}.\]
Thus equation \ref{eq1} can be expressed as:
\begin{equation} \label{eq2}
0 = (1 + ab - a - b)\psi^2 + (a + b - 1 - P(0,0,1))\psi.
\end{equation}
Here, let $a_{II}=1 + ab - a - b$, $b_{II}=a + b - 1 - P(0,0,1)$, from equation \ref{eq2} we know $a_{II}\psi + b_{II} = 0$, thus we have
\begin{eqnarray*}
\psi &=& -\frac{b_{II}}{a_{II}}\\
&=& \frac{1 - P(B(1) = 0) - P(B(2) = 0) + P(B(1:2) = (0,0))}{1 + P(B(1) = 0)P(B(2)= 0) - P(B(1) = 0) - P(B(2) = 0)}
\end{eqnarray*}
where 
$P(B(1) = 0) = P(0,1,0) + P(0,1,1) + P(0,0,1), P(B(2) = 0) = P(1,0,0) + P(1,0,1) + P(0,0,1),$ and $P(B(1:2) = (0,0)) = P(0,0,1)$
\end{proof}

\subsection{Efficient influence curve under independence assumption}\label{apd_indep_d}
\begin{lemma}
For constraint $\Phi_{II} = 0$(equation \ref{PhiII}), we have the efficient influence curve $D^*_{\Phi_{II}}(P)$ as:
\begin{eqnarray}
D^*_{\Phi_{II}}(P) &=& \frac{\partial \psi}{\partial P(B(1) = 0)}\times D^*_{\Phi_{II}}(P(B(1) = 0)) \nonumber \\
&&+ \frac{\partial \psi}{\partial P(B(2) = 0)}\times D^*_{\Phi_{II}}(P(B(2) = 0)) \nonumber \\
&&+ \frac{\partial \psi}{\partial P(B(1:2) = (0,0))}\times D^*_{\Phi_{II}}(P(B(1:2) = (0,0))).\label{lm2}
\end{eqnarray}
\end{lemma}
\begin{proof}
By the delta method \cite{doob1935limiting}, the efficient influence curve of $\psi$ can be written as a function of each components' influence curve. The efficient influence curves of the three components are presented as follows: \\
\begin{eqnarray*}
D^*_{\Phi_{II}}(P(B(1) = 0)) &=& \mathbbm{I}(B(1) = 0) - P(B(1) = 0).\\
D^*_{\Phi_{II}}(P(B(2) = 0)) &=& \mathbbm{I}(B(2) = 0) - P(B(2) = 0).\\
D^*_{\Phi_{II}}(P(B(1:2) = (0,0))) &=& \mathbbm{I}(B(1:2) = (0,0)) - P(B(1:2) = (0,0)).
\end{eqnarray*}
Therefore, we only need to calculate the derivatives to get $D^*_{\Phi_{II}}(P)$, and the three parts of derivatives are given by
\begin{equation*}
\frac{\partial \psi}{\partial P(B(1) = 0)} = \frac{(1 - P(B(2) = 0))(P(B(1:2) = (0,0)) - P(B(1) = 0))}{(1 - P(B(1) = 0) - P(B(2) = 0) + P(B(1) = 0)P(B(2) = 0))^2}.
\end{equation*}

\begin{equation*}
\frac{\partial \psi}{\partial P(B(2) = 0)} = \frac{(1 - P(B(1) = 0))(P(B(1:2) = (0,0)) - P(B(1) = 0))}{(1 - P(B(1) = 0) - P(B(2) = 0) + P(B(1) = 0)P(B(2) = 0))^2}.
\end{equation*}
\begin{equation*}
\frac{\partial \psi}{\partial P(B(1:2) = (0,0))} = \frac{1}{1 - P(B(1) = 0) - P(B(2) = 0) + P(B(1) = 0)P(B(2) = 0)}.
\end{equation*}
Plug each part into equation \ref{lm2}, we have
\begin{eqnarray*}
D^*_{\Phi_{II}}(P)
&=& \frac{(1 - P(B(2) = 0))(P(B(1:2) = (0,0)) - P(B(2) = 0))}{(1 - P(B(1) = 0) - P(B(2) = 0) + P(B(1) = 0)P(B(2) = 0))^2}\\
&& [\mathbbm{I}(B(1) = 0) - P(B(1) = 0)]\\
&&+ \frac{(1 - P(B(1) = 0))(P(B(1:2) = (0,0)) - P(B(1) = 0))}{(1 - P(B(1) = 0) - P(B(2) = 0) + P(B(1) = 0)P(B(2) = 0))^2}\\
&& [\mathbbm{I}(B(2) = 0) - P(B(2) = 0)]\\
&&+ \frac{1}{1 - P(B(1) = 0) - P(B(2) = 0) + P(B(1) = 0)P(B(2) = 0)}\\
&& [\mathbbm{I}(B(1:2) = (0,0)) - P(B(1:2) = (0,0))].
\end{eqnarray*}
\end{proof}

\subsection{Efficient influence curve under multidimensional constraint}\label{apd_lemma}
\begin{lemma}
Consider a model ${\cal M}\equiv \{P\in {\cal M}_1:\Phi(P)=0\}$ defined by an initial larger model ${\cal M}_1$ and multivariate  constraint function $\Phi:{\cal M}_1\rightarrow\openr^K$.
Suppose that $\Phi:{\cal M}_1\rightarrow\openr^K$ is path-wise differentiable at $P$ with efficient influence curve $D^*_{\Phi}(P)$ for all $P\in {\cal M}_1$.
Let $T_1(P)$ be the tangent space at $P$ for model ${\cal M}_1$, and let $\Pi_{T_1}:L^2_0(P)\rightarrow T_1(P)$ be the projection operator onto $T_1(P)$.
The tangent space at $P$ for model ${\cal M}$ is given by:
\[
T(P)=\{S\in T_1(P): S\perp D^*_{\Phi}(P)\}.\]
The projection onto $T(P)$ is given by:
\[
\Pi_{T}(S)=\Pi_T(S)-\Pi_{D^*_{\Phi}}(\Pi_T(S)),\]
where $\Pi_{D^*_{\Phi}}$ is the projection operator on the $K$-dimensional subspace of $T_1(P)$ spanned by the components of $D^*_{\Phi}(P)$. The latter projection is given by the formula:
\[
\Pi_{D^*_{\Phi}}(S)= E(S D^*_{\Phi}(P)^{\top}) E (D^*_{\Phi}(P)D^*_{\Phi}(P)^{\top})^{-1}D^*_{\Phi}(P).\]
\end{lemma}

\subsection{Target parameter under conditional independence assumption}\label{apd_condInf_psi}
\begin{lemma}
For constraint $\Phi_{CI} = 0$ (equation \ref{cond_indep_constraint}), we have the target parameter $\Psi_{CI}$ as
\begin{equation*}
\Psi_{CI} = \frac{P(B_m = 1, B_j = 1, 0,...,0)}{C_0(P)}.
\end{equation*}
where 
\begin{eqnarray*}
C_0(P)&=& P(B_m = 1, B_j = 1,  0,...,0)\\
&& + P(B_m = 0, B_j = 1,  0,...,0)P(B_m = 1, B_j = 0,  0,...,0).
\end{eqnarray*}
\end{lemma}

\begin{proof}
The constrain $\Phi_{CI} = 0$(equation \ref{cond_indep_constraint}) can be written as
\begin{eqnarray*}
P^*(B_j =1 \vert B_1 = 0, \cdots, B_m = 0, \cdots, B_K = 0) &=&\\ 
P^*(B_j =1 \vert B_1 = 0, \cdots, B_m = 1, \cdots, B_K = 0).
\end{eqnarray*}
This equation is equivalent to
\begin{eqnarray}
P^*(0,\dots,0) &=& \frac{P^*(B_m = 0, B_j = 1, 0,\dots,0)}{P^*(B_m = 1, B_j = 1, 0,\dots,0)}\nonumber \\
&& \times P^*(B_m = 1, B_j = 0, 0,\dots,0).\label{CI}
\end{eqnarray}
As $\Psi_{CI} \equiv 1 - P^*(0,\dots,0)$, we have 
\[P^*(B_m = 1, B_j = 0, 0,\dots,0) = P(B_m = 1, B_j = 0, 0,\dots,0)\Psi_{CI}.\] 
\[P^*(B_m = 0, B_j = 1, 0,\dots,0) = P(B_m = 0, B_j = 1, 0,\dots,0)\Psi_{CI}.\] 
\[P^*(B_m = 1, B_j = 1, 0,\dots,0) = P(B_m = 1, B_j = 1, 0,\dots,0)\Psi_{CI}.\] 
Therefore, equation \ref{CI} is equivalent to 
\[\Psi_{CI} = \frac{P(B_m = 1, B_j = 1, 0,\dots,0)}{P(B_m = 1, B_j = 1, 0,\dots,0) + C_9}.
\]
where 
\[C_9 = P(B_m = 0, B_j = 1, 0,\dots,0)P(B_m = 1, B_j = 0, 0,\dots,0).\]
\end{proof}

\subsection{Efficient influence curve under conditional independence assumption}\label{apd_condInf_d}
\begin{lemma}\label{lemma_CI_d}
For constraint $\Phi_{CI} = 0$ (equation \ref{cond_indep_constraint}), we have the efficient influence curve $D^*_{\Phi_{CI}}(P)$ as
\begin{eqnarray*}
D^*_{\Phi_{CI}}(P) &=& \frac{1}{C_5(P)}(C_6(P) - C_7(P) - C_8(P))
\end{eqnarray*}
where
\begin{eqnarray*}
C_5(P)&=& -\sum_{b_1\not =0,b_2\not =0}I(b_1(1)=b_2(2)=0)P(b_1)P(b_2).\\
C_6(P)&=& P(B_m = 0, B_j = 1, 0,...,0)P(B_m = 0, B_j = 0, 0,...,0)\\
&& [\mathbbm{I}(B_m = 1, B_j = 1, 0,...,0) - P(B_m = 1, B_j = 1, 0,...,0)].\\
C_7(P) &=& P(B_m = 1, B_j = 0, 0,...,0)P(B_m = 1, B_j = 1, 0,...,0)\\
&& [\mathbbm{I}(B_m = 0, B_j = 1, 0,...,0) - P(B_m = 0, B_j = 1, 0,...,0)].\\
C_8(P) &=& P(B_m = 0, B_j = 1, 0,...,0)P(B_m = 1, B_j = 1, 0,...,0)\\
&& [\mathbbm{I}(B_m = 1, B_j = 0, 0,...,0) - P(B_m = 1, B_j = 0, 0,...,0)].
\end{eqnarray*}
\end{lemma}

\begin{proof}
By the delta method \cite{doob1935limiting}, the efficient influence curve of $\Psi_{CI}$ can be written as a function of each components' influence curve. The efficient influence curves of the three components are presented as follows \\
\begin{eqnarray}
D^*_{\Phi_{CI}}(P(B_m = 1, B_j = 1, 0,\dots,0)) &=& \mathbbm{I}(B_m = 1, B_j = 1, 0,\dots,0)\nonumber \\
&& - P(B_m = 1, B_j = 1, 0,\dots,0).\\
D^*_{\Phi_{CI}}(P(B_m = 0, B_j = 1, 0,\dots,0)) &=& \mathbbm{I}(B_m = 0, B_j = 1, 0,\dots,0)\nonumber \\
&& - P(B_m = 0, B_j = 1, 0,\dots,0).\\
D^*_{\Phi_{CI}}(P(B_m = 1, B_j = 0, 0,\dots,0)) &=& \mathbbm{I}(B_m = 1, B_j = 0, 0,\dots,0)\nonumber \\
&& - P(B_m = 1, B_j = 0, 0,\dots,0)).\label{D_CI}
\end{eqnarray}
Therefore, we only need to calculate the derivatives to get $D^*_{\Phi_{CI}}(P)$, and the three parts of derivatives are given by
\begin{eqnarray*}
\frac{\partial \psi}{\partial P(B_m = 1, B_j = 1, 0,\dots,0)} &=& \frac{P(B_m = 0, B_j = 1, 0,\dots,0)}{C_4}\\
&& \times P(B_m = 1, B_j = 0, 0,\dots,0).\\
\frac{\partial \psi}{\partial P(B_m = 0, B_j = 1, 0,\dots,0)} &=& - \frac{P(B_m = 1, B_j = 1, 0,\dots,0)}{C_4}\\
&& \times P(B_m = 1, B_j = 0, 0,\dots,0).\\
\frac{\partial \psi}{\partial P(B_m = 1, B_j = 0, 0,\dots,0)} &=& - \frac{P(B_m = 1, B_j = 1, 0,\dots,0)}{C_4}\\
&& \times P(B_m = 1, B_j = 1, 0,\dots,0).
\end{eqnarray*}
where 
\begin{eqnarray*}
C_4 &=& [P(B_m = 1, B_j = 1, 0,\dots,0) + P(B_m = 0, B_j = 1, 0,\dots,0)\\
&& \times P(B_m = 1, B_j = 0, 0,\dots,0)]^2.
\end{eqnarray*}
Plug each part into equation \ref{D_CI}, we have that lemma \ref{lemma_CI_d} is true.
\end{proof}

\subsection{Asymptotic efficiency of TMLE}\label{apd_efficincy_TMLE}
In this section we prove the following theorem \ref{thm_asy_eff} establishing asymptotic efficiency of the TMLE. The empirical NPMLE is also efficient since asymptotically all cells are filled up. And TMLE is just a finite sample improvement and asymptotically TMLE acts as the empirical NPMLE. The estimators $P_n*$ will be like parametric model MLE and thus converge at rate $o_P(1/\sqrt{n})$. 
\begin{theorem}
Consider the TMLE $\Psi_I(P_n^*)$ of $\Psi_I(P_0)$ defined above, satisfying $P_n D^*_{\Phi_I}(P_n^*)=o_P(1/\sqrt{n})$.
We assume $P_0(b)>0$ for all $b\not =0$.
If $\parallel P_n^*-P_0\parallel =o_P(n^{-1/4})$, then $\Psi_I(P_n^*)$ is an asymptotically efficient estimator of $\Psi_I(P_0)$:
\[
\Psi_I(P_n^*)-\Psi_I(P_0)=(P_n-P_0)D^*_{\Phi_I}(P_0)+o_P(1/\sqrt{n}).\]
\label{thm_asy_eff}
\end{theorem}
\begin{proof}

Define $f_{I,0}(b)=f_I(b)+f_I(0)$.Note that
\[
P_0 D^*_{\Phi_I}(P)=(-1)^K \psi_I(1-\psi_I)\sum_{b\not =0}f_{I,0}(b)\frac{(P_0-P)(b)}{P(b)}.\]
Consider the second order Taylor expansion of $\Psi_I(P_0)$ at $P$:
\[
\Psi_I(P_0)-\Psi_I(P)=d\Psi_I(P)(P_0-P)+R_2(P,P_0),\]
where
\begin{eqnarray*}
d\Psi_I(P)(h)&=&\left . \frac{d}{d\epsilon}\Psi_I(P+\epsilon h) \right |_{\epsilon =0}\\
&=&(-1)^K\Psi_I(P)(1-\Psi_I(P))\sum_{b\not =0}\frac{f_I(b)h(b)}{P(b)},\end{eqnarray*}
\begin{eqnarray}
R_2(P, P_0) &=& \left . \frac{1}{2} \frac{d^2 \Psi_I(P+\epsilon h)}{d\epsilon^2} (P_0 - P)^2 \right |_{\epsilon =0} + o_P((P_0-P)^2)\nonumber \\
&& =  \frac{1}{2}\Psi_I(P)(1-\Psi_I(P))(P_0 - P)^2(b)\times \nonumber \\
&& [(1-2\Psi_I(P))(\sum_{b\not =0} \frac{f_I(b)h(b)}{P(b)})^2 + (-1)^{K+1} \sum_{b\not =0} \frac{f_I(b)h(b)^2}{P(b)^2}] \nonumber \\
&& + o_P((P_0-P)^2(b))\label{R2}
\end{eqnarray}
From equation \ref{R2} we know that $R_2(P,P_0)$ is a second order term involving square differences $(P_0-P)^2(b)$ for $b\not =0$.
Thus, we observe that 
\[
P_0 D^*_{\Phi_I}(P)=d\Psi_I(P)(P_0-P).\]
This proves that
\[
P_0 D^*_{\Phi_I}(P)=\Psi_I(P_0)-\Psi_I(P)-R_2(P,P_0).\]

We can apply this identity to $P_n^*$ so that we obtain $P_0 D^*_{\Phi_I}(P_n^*)=\Psi_I(P_0)-\Psi_I(P_n^*)-R_2(P_n^*,P_0)$.
Combining this identity with $P_n^*D^*_{\Phi_I}(P_n^*)=0$ yields:
\[
\Psi_I(P_n^*)-\Psi_I(P_0)=(P_n-P_0)D^*_{\Phi_I}(P_n^*)+R_2(P_n^*,P_0).\]

We will assume that $P_n^*$ is consistent for $P_0$ and $P_0(b)>0$ for all $b\not =0$, so that it follows that 
$D^*_{\Phi_I}(P_n^*)$ is uniformly bounded  by a $M<\infty$ with probability tending to 1, an that it  falls in a $P_0$-Donsker class (dimension is finite $2^K-2$), and
$P_0\{ D^*_{\Phi_I}(P_n^*)-D^*_{\Phi_I}(P_0)\}^2\rightarrow 0$ in probability as $n\rightarrow\infty$.
By empirical process theory, it now follows that 
\[ (P_n-P_0)D^*_{\Phi_I}(P_n^*)=(P_n-P_0)D^*_{\Phi_I}(P_0)+o_P(1/\sqrt{n}).\]
We also note that $R_2(P_n^*,P_0)$ has denominators that are bounded away from zero, so that  $R_2(P_n^*,P_0)=o_P(1/\sqrt{n})$ if 
$\parallel P_n^*-P_0 \parallel =o_P(n^{-1/4})$ (e.g., Euclidean norm). Thus theorem \ref{thm_asy_eff} is proved.
\end{proof}

 \section{Acknowledgement}
This research was supported by NIH grant R01AI125842 and NIH grant 2R01AI074345.

 \end{document}